\documentclass{elsarticle}
\usepackage{mathrsfs,amsthm,mathtools,graphicx,epstopdf,stmaryrd,subfigure,float,amsmath,bm,amsfonts,amssymb,natbib,dsfont,amsfonts,tikz,fancyhdr,multirow,hyperref,caption,color}
\usepackage[top=1.25in, bottom=1.0in, left=1.0in, right=1.0in]{geometry}
\usepackage{lineno}

\numberwithin{equation}{section}
\newtheorem{thm}{Theorem}[section]

\newtheorem{rmk}{Remark}[section]
\newtheorem{lem}{Lemma}[section]
\newtheorem{ass}{Assumption}[section]

\newcommand{\jp}[1]{\ensuremath{\displaystyle[\![#1]\!]}}
\newcommand{\avg}[1]{\ensuremath{\displaystyle\{#1\}}}
\newcommand{\NNorm}[1]{{\left\vert\kern-0.25ex\left\vert\kern-0.25ex\left\vert #1 \right\vert\kern-0.25ex\right\vert\kern-0.25ex\right\vert}}
\newcommand{\QK}{\mathbf{Q}_{K}^{1}}
\newcommand{\Idiv}{\mathcal I_h^{\rm div}}
\newcommand{\Pzero}{P_h^0}
\newcommand{\alphaG}{\alpha_{\Gamma}}
\newcommand{\Gdiv}{G_{\operatorname{div}}}
\newcommand{\Gker}{G_{\operatorname{ker}}}
\newcommand{\Mgam}{\mathcal M_{\Gamma}}

\newcommand{\Pigam}{\Pi_{\Gamma}^{0}}
\newcommand{\PF}{P_{\mathcal F}}

\begin{document}

	\begin{frontmatter}
		\title{A Mixed Extended Virtual Element Method for Elliptic Interface Problems on Polygonal Meshes}
		
		\address[a]{School of Science, Jinling Institute of Technology, Nanjing 211169, China}
		\address[b]{Ministry of Education Key Laboratory for NSLSCS, Jiangsu Collaborative Innovation Center of Biomedical Functional Materials, School of Mathematical Sciences, Nanjing Normal University, Nanjing 210023, China} 
	\address[c]{School of Mathematical Sciences, Jiangsu Second Normal University, Nanjing 211200, China}
	\cortext[mycorrespondingauthor]{Corresponding author}
		\author[a]{Xianyan Zheng}
	\ead{xyZheng@jit.edu.cn}
		\author[b,c]{Jinru Chen\corref{mycorrespondingauthor}}
	\ead{jrchen@njnu.edu.cn}
		\author[b]{Feng Wang}
	\ead{fwang@ninu.edu.cn}
	
	\begin{abstract}
		We propose a lowest-order \(H(\operatorname{div})\)-conforming mixed extended virtual element method for elliptic interface problems on interface-unfitted polygonal meshes. The flux and pressure are approximated by subdomain-wise extended \(H(\operatorname{div})\)-VEM spaces and by piecewise constants, respectively. On cut elements, the computable polynomial projection is defined on the whole background element and then restricted to the two subdomains. Compared with BDM-type polynomial spaces, the mixed VEM space contains a non-polynomial component, which gives rise to additional
		consistency terms on cut elements. To control these terms, we use an enhanced kernel stabilization on cut elements and an interface normal-flux average in the mixed coupling. A corrected interface-flux penalty and a local divergence
		ghost penalty are added to obtain cut-position-independent stability without using a volume div-div augmentation. We prove continuity, a discrete inf-sup condition, and an optimal first-order error estimate in a mesh-dependent norm. The constants are independent of the mesh size and of the position of the interface relative to the background mesh, but may depend on the coefficient contrast.
	\end{abstract}
	
\end{frontmatter}

\section{Introduction}

Interface problems arise in many models with discontinuous material parameters, including heat conduction in composite media, flow in heterogeneous porous media, and transmission problems in mechanics and electromagnetics. Their numerical approximation is difficult because the solution is usually only piecewise smooth and suitable transmission conditions have to be imposed on the solution and on the normal flux across the interface. Interface-fitted finite element methods can recover optimal accuracy under suitable regularity assumptions, but fitted mesh generation may be expensive when the interface is curved, complicated, or moving; see, for instance, the classical convergence analysis in \cite{ChenZou1998}.

Unfitted methods avoid the generation of interface-fitted meshes. The extended finite element method was originally introduced for cracks and discontinuities \cite{BelytschkoBlack1999,MoesDolbowBelytschko1999} and later developed into a broad class of generalized finite element methods \cite{FriesBelytschko2010}. Nitsche's method \cite{Nitsche1971} provides a natural weak enforcement of boundary and interface conditions, and the Nitsche-XFEM of Hansbo and Hansbo \cite{HansboHansbo2002} is one of the standard unfitted approaches for elliptic interface problems. Related unfitted and immersed methods include immersed interface and immersed finite element methods \cite{Li1998,LiLinWu2003}, discontinuous Galerkin methods for interface problems \cite{CaiYeZhang2011}, and CutFEM methods with ghost penalties \cite{Burman2010Ghost,BurmanEtAl2015CutFEM,HansboLarsonZahedi2014,KirchhartGrossReusken2016}.

In the formulation considered here, we use an \(H(\operatorname{div})\)-conforming approximation of the flux. On cut
elements, the interface term involves the normal component of the test flux on \(\Gamma_K\). Since virtual element functions are not known explicitly inside the element, this normal flux information has to be recovered from the degrees of freedom, and the resulting coupling with the piecewise constant pressure has to be stable. For elliptic interface problems, an extended mixed finite element method based on the BDM\(_1\)--\(P_0\) pair was developed in \cite{CAO2022148}; the method uses suitable stabilizations to obtain an inf-sup condition and optimal convergence independently of the position of the interface relative to the background triangulation. Related unfitted mixed and coupled formulations can be found in \cite{DAngeloScotti2012,Ji2022RT,CaoChen2022DarcyStokes}.

The virtual element method (VEM) provides a flexible framework for polygonal and polyhedral meshes. Its local basis functions are not required to be known explicitly; instead, the method is built from degrees of freedom, polynomial projections, and stabilization terms. The basic principles of the conforming VEM were introduced in \cite{BeiraoBrezziCangiani2013}, and further developments include general second-order elliptic problems \cite{BeiraoBrezziMariniRusso2016SecondOrder}, stability analyses of VEM stabilizations \cite{BeiraoLovadinaRusso2017}, VEMs on meshes with small edges or faces \cite{BrennerSung2018}, and nonconforming virtual elements \cite{AyusoLipnikovManzini2016}. Mixed VEMs were introduced as polygonal counterparts of classical mixed finite elements \cite{BrezziFalkMarini2014}, building on the theory of mixed finite elements \cite{BrezziDouglasMarini1985,BrezziFortin1991}. A systematic construction of \(H(\operatorname{div})\)- and \(H(\operatorname{curl})\)-conforming virtual element spaces, together with the computability of the associated \(L^2\)-projections, was presented in \cite{2016H(div)H(curl)}.

Several extended or immersed VEMs have been proposed for interface-unfitted problems. Extended VEMs based on enrichment
techniques were proposed for singularities and discontinuities \cite{BenvenutiChiozziManzini2019}, while
conforming, nonconforming, and high-order extended VEMs have been studied for elliptic interface problems on interface-unfitted meshes \cite{ZhengChenWang2024XVEM,ZhengChenWang2024NCXVEM}. Related immersed
VEM constructions have also been proposed  \cite{CaoChenGuoLin2022IVEM}. The precise additional terms depend on the chosen space and formulation: for example, nonconforming
extended VEMs require special edge terms and stabilization terms, whereas high-order extended VEMs require additional consistency and stabilization treatments.


The present work develops a mixed extended VEM for elliptic interface problems on interface-unfitted polygonal meshes. The flux is approximated in subdomain-wise extended \(H(\operatorname{div})\)-conforming VEM spaces, while the pressure is approximated by piecewise constants. The polynomial projection of a virtual velocity is computed on the whole background element and is then restricted to the two cut subdomains. This convention keeps the method computable on polygonal meshes, but it creates additional consistency terms in the mixed formulation.

A key difference from BDM-based extended mixed finite elements is the presence of a non-polynomial component in the \(H(\operatorname{div})\)-VEM velocity space. In BDM spaces the polynomial projection is the identity on the discrete velocity. In VEM, however,
\[
\mathbf v_h=\mathbf Q_K^1\mathbf v_h+(I-\mathbf Q_K^1)\mathbf v_h,
\]
and the kernel component \((I-\mathbf Q_K^1)\mathbf v_h\) is not explicitly known. On cut elements this component appears in the consistency error because the computable projection is defined on the whole background element whereas the integrals are taken on the cut subdomains. These terms are absent in BDM-type polynomial spaces and require a VEM-specific treatment.

The method proposed in this paper is designed to address these issues without adding a volume div-div augmentation. The pressure stabilizations \(J_1\) and \(J_2\) are used in the inf-sup proof: \(J_1\) controls jumps of the piecewise constant pressure across transmission edges associated with the same subdomain, while \(J_2\) controls the pressure jump across the interface. The local divergence ghost penalty gives an extended-domain estimate for the discrete divergence. An enhanced stabilization is used on cut elements to control the kernel component \((I-\mathbf Q_K^1)\mathbf v_h\).

For the interface term in the mixed coupling, only the average of the normal flux over each interface segment is needed, because the pressure is piecewise constant on each side of the segment. We therefore introduce an interface normal-flux average, denoted by \(\mathcal M_\Gamma\mathbf v_h\). This quantity is a computable constant on each interface segment whose integral equals the integral of the true normal flux over that segment. For stability and consistency estimates, we also control the zero-mean part of the projected polynomial normal flux. This is done through a corrected interface flux \(\mathcal F_\Gamma\) and the corresponding penalty term.

We prove the continuity of the discrete bilinear form, a discrete inf-sup condition, and an optimal first-order error estimate in a mesh-dependent norm. The proof separates the roles of the different terms. The pressure stabilizations \(J_1\) and \(J_2\) are used in the inf-sup proof: \(J_1\) controls jumps of the piecewise constant pressure across transmission edges associated with the same subdomain, while \(J_2\) controls the pressure jump across the interface. The divergence ghost penalty \(G_{\rm div}\) gives an extended-domain estimate for the discrete divergence: the divergence norm on the subdomain-wise extended domains is bounded by the corresponding physical-domain norm together with \(G_{\rm div}\). The enhanced stabilization and the corrected interface-flux penalty control the consistency terms caused by the kernel component of the VEM velocity on cut elements. The constants in the stability and error estimates are independent of the mesh size and of the position of the interface relative to the background mesh, but they may depend on the contrast of the diffusion coefficient.

The rest of the paper is organized as follows. Section~2 introduces the model problem and its mixed weak formulation. Section~3 presents the mesh notation, the local \(H(\operatorname{div})\) virtual element space, and the discrete formulation. Section~4 proves the auxiliary estimates, continuity, pressure lifting, and the discrete inf-sup condition. Section~5 establishes the interpolation estimates and the final error bound. Section~6 reports numerical experiments, and Section~7 gives concluding remarks.

\section{The model problem and weak formulation}

Consider the following elliptic interface problem
\begin{equation}\label{model problem}
	\left\{
	\begin{aligned}
		-\operatorname{div}(\alpha \nabla p) &= f \quad \operatorname{in} ~\Omega_1 \cup \Omega_2,\\
		\jp{ \alpha \nabla p \cdot \mathbf{n} }  &= 0 \quad \operatorname{on}~\Gamma,\\
		\jp{ p } &= 0 \quad \operatorname{on}~\Gamma,\\
		p &= 0 \quad \operatorname{on}~\partial \Omega.
	\end{aligned}
	\right.
\end{equation}
where $\jp{p} = p_1|_{\Gamma} - p_2|_{\Gamma}$, $p_i=p|_{\Omega_i}$, $i=1,2$, and $\mathbf n$ is the unit normal vector on $\Gamma$ pointing from $\Omega_1$ to $\Omega_2$.

Set $\mathbf{u} = \alpha \nabla p$. Then
\begin{equation}\label{mixed form of model problem}
	\left\{
	\begin{aligned}
		-\operatorname{div}\mathbf{u} &= f \qquad~\operatorname{in} ~\Omega_1 \cup \Omega_2,\\
		\mathbf{u} &= \alpha \nabla p \quad \operatorname{in} ~\Omega_1 \cup \Omega_2,\\
		\jp{ \mathbf{u} \cdot \mathbf{n} }  &= 0 \qquad~\operatorname{on}~\Gamma,\\
		\jp{ p } &= 0 \qquad~\operatorname{on}~\Gamma,\\
		p &= 0 \qquad~\operatorname{on}~\partial \Omega.
	\end{aligned}
	\right.
\end{equation}
The weak formulation of \eqref{mixed form of model problem} is to find
$(\mathbf{u},p) \in \mathbf{V} \times Q = H(\operatorname{div};\Omega) \times L^2(\Omega)$ such that
\begin{equation}\label{weak formulation of mixed problem}
	\left\{
	\begin{aligned}
		&(\alpha^{-1} \mathbf{u},\mathbf{v})_{\Omega} + (p,\operatorname{div} \mathbf{v})_{\Omega} = 0 \quad \forall \mathbf{v} \in \mathbf{V},\\
		&-(q,\operatorname{div} \mathbf{u})_{\Omega} = (f,q)_{\Omega} \qquad \forall q \in Q.
	\end{aligned}
	\right.
\end{equation}
It is well-known that \eqref{weak formulation of mixed problem} is well-posed.

\section{Discrete space and bilinear form}

\subsection{Local discrete space}
For $k \geq 0$, let $\mathbb{P}_k(D)$ be the set of polynomials on $D$ of degree at most $k$, with $D=e$ or $K$. Define
\begin{equation*}
	\mathcal{G}_1(K) := \nabla (\mathbb{P}_2(K)) \subseteq [\mathbb{P}_1(K)]^2,
	\qquad
	\mathcal{G}_1(K)^{\perp}\subseteq [\mathbb{P}_1(K)]^2
\end{equation*}
as the $L^2(K)$-orthogonal complement of $\mathcal G_1(K)$.
On each element $K \in \mathcal{T}_h$, define
\begin{equation}\label{local VEM space}
	\mathbf{V}_h(K) = \{\mathbf{v} \in H(\operatorname{div};K) \cap H(\operatorname{rot};K):
	\mathbf{v} \cdot \mathbf{n}|_e \in \mathbb{P}_1(e)~ \forall e \subset \partial K,
	~ \operatorname{div} \mathbf{v} \in \mathbb{P}_0(K),
	~ \operatorname{rot} \mathbf{v} \in \mathbb{P}_0(K)\}.
\end{equation}
We use scaled degrees of freedom. Let $\{m_{e,0},m_{e,1}\}$ be a scaled basis of $\mathbb P_1(e)$ and let $\mathbf g_K^\perp$ be a scaled generator of $\mathcal G_1(K)^\perp$. For $\mathbf v\in\mathbf V_h(K)$, set
\begin{equation}\label{scaled dofs}
	\widehat\chi_{e,\ell}(\mathbf v)=\frac{1}{|e|}\int_e (\mathbf v\cdot \mathbf n)m_{e,\ell}\,ds,
	\qquad \ell=0,1,\quad e\subset \partial K,
\end{equation}
and
\begin{equation}\label{scaled internal dof}
	\widehat\chi_K(\mathbf v)=\frac{1}{|K|}\int_K \mathbf v\cdot \mathbf g_K^\perp\,dx.
\end{equation}
We enumerate these scaled degrees of freedom by $\widehat\chi_r$, $r=1,\ldots,N_K$, where $N_K=2n_K+1$ and $n_K$ is the number of edges of $K$. They are scaled versions of the standard moment degrees of freedom for the lowest-order $H(\operatorname{div})$ VEM space and are unisolvent; see \cite{2016H(div)H(curl)}.

\subsection{Global discrete space and averages on cut cells}
Let
\begin{equation*}
	\mathcal T_h^\Gamma=\{K\in\mathcal T_h:K\cap\Gamma\ne\emptyset\},
	\qquad K_i=K\cap\Omega_i,
	\qquad
	\mathcal T_{h,i}^{+}=\{K\in\mathcal T_h:K_i\ne\emptyset\}.
\end{equation*}
Set $\Omega_{h,i}=\cup_{K\in\mathcal T_{h,i}^{+}}K$, $i=1,2$, and define
\begin{equation*}
	\begin{aligned}
		\mathbf{V}_{h,i} &= \{\mathbf{v}_h \in H(\operatorname{div};\Omega_{h,i}): \mathbf{v}_h|_K \in \mathbf{V}_h(K)~ \forall K \in \mathcal{T}_{h,i}^{+} \},\\
		Q_{h,i} &= \{q_h \in L^2(\Omega_{h,i}): q_h|_K \in \mathbb{P}_0(K)~ \forall K \in \mathcal{T}_{h,i}^{+}\},\\
		\mathbf{V}_h &= \mathbf{V}_{h,1} \times \mathbf{V}_{h,2}, \qquad Q_h = Q_{h,1} \times Q_{h,2}.
	\end{aligned}
\end{equation*}

For the subdomain-wise ghost penalties, we use the set of transmission edges
\begin{equation}\label{ghost edge set}
	\mathcal E_{\Gamma,i}=\{e=\partial K_l\cap\partial K_r:K_l,K_r\in\mathcal T_{h,i}^{+},~ e\subset\Omega_{h,i},~ K_l\hbox{ or }K_r\in\mathcal T_h^\Gamma\},
\end{equation}
possibly enlarged by a uniformly bounded number of neighboring layers when needed. This is the standard transmission-edge set used to obtain the extended domain estimates below.

\begin{ass}[Mesh and interface geometry]\label{ass unfitted geometry}
	The background mesh and the interface satisfy the following standard assumptions.
	\begin{enumerate}
		\item[(A1)] Each element \(K\in\mathcal T_h\) is a simple polygon.
		
		\item[(A2)] The polygonal mesh is shape regular in the standard VEM sense: each element \(K\) is star-shaped with respect to a disk whose diameter is comparable to \(h_K\), and the number of edges of each element is uniformly bounded.
		
		\item[(A3)] For sufficiently small \(h\), the interface \(\Gamma\) intersects the boundary of each cut element in exactly two points, and each open edge is intersected at most once.
		
		\item[(A4)] Let \(\Gamma_{K,h}\) be the straight line segment joining the two intersection points of \(\Gamma\) with \(\partial K\). In a local coordinate system such that
		\[
		\Gamma_{K,h}=\{(\xi,\eta):0<\xi<|\Gamma_{K,h}|,\ \eta=0\},
		\]
		the curve segment \(\Gamma_K=\Gamma\cap K\) can be represented as
		\[
		\Gamma_K=\{(\xi,\eta):0<\xi<|\Gamma_{K,h}|,\ \eta=\delta_K(\xi)\},
		\]
		where \(\delta_K\) is a smooth function.
		
		\item[(A5)] For each \(i=1,2\), every cut element in the extended mesh \(\mathcal T_{h,i}^{+}\) can be connected, through a uniformly bounded number of transmission edges in \(\mathcal E_{\Gamma,i}\), to an element
		\(K'\in\mathcal T_{h,i}^{+}\) satisfying
		\[
		|K'\cap\Omega_i|\ge c_0|K'|.
		\]
		The lengths of these chains and their overlaps are uniformly bounded.
	\end{enumerate}
	Assumptions (A1)--(A2) are the standard shape-regularity conditions used for polygonal VEM spaces. Assumptions (A3)--(A4) follow the usual local geometry assumptions in interface-unfitted methods. Assumption (A5) is the local patch condition used in ghost-penalty arguments. It ensures that quantities defined on the subdomain-wise extended meshes can be controlled by their physical-domain contributions together with the corresponding ghost penalties. We use the resulting extension estimates for piecewise constants and for the discrete divergence below.
\end{ass}

For $K\in\mathcal T_h^\Gamma$ we use the harmonic Nitsche weights
\begin{equation}\label{harmonic weights}
	\omega_1=\frac{\alpha_2}{\alpha_1+\alpha_2},\qquad
	\omega_2=\frac{\alpha_1}{\alpha_1+\alpha_2},\qquad
	\alphaG=\frac{2\alpha_1\alpha_2}{\alpha_1+\alpha_2}.
\end{equation}
Thus $\alphaG^{-1}=\frac12(\alpha_1^{-1}+\alpha_2^{-1})$. On $\Gamma_K$ we use the weighted average
\begin{equation}\label{average definition}
	\avg{r}:=\omega_1 r_1+\omega_2 r_2.
\end{equation}
The jump is $\jp{r}=r_1-r_2$ for scalar quantities and
$\jp{\mathbf v\cdot\mathbf n}=\mathbf v_1\cdot\mathbf n-\mathbf v_2\cdot\mathbf n$ for normal fluxes.

\subsection{Projectors, extensions, and interpolation}
Let $\QK:L^2(K)^2\rightarrow [\mathbb P_1(K)]^2$ be the $L^2(K)$-projection:
\begin{equation}\label{definition QK}
	(\QK\mathbf v,\mathbf q)_K=(\mathbf v,\mathbf q)_K\qquad\forall \mathbf q\in [\mathbb P_1(K)]^2.
\end{equation}
For virtual functions this projection is computable from the degrees of freedom. On a cut element, $\QK\mathbf v$ is first computed on the whole background element $K$ and is then restricted to $K_i$ or $\Gamma_K$.

We use the following extension and interpolation operators in the analysis. For each subdomain $\Omega_i$, let
\[
E^i:H^m(\Omega_i)\to H^m(\Omega_{h,i}),\qquad m=1,2,
\]
be a linear extension operator satisfying
\begin{equation}\label{scalar extension bound}
	\|E^i\phi\|_{m,\Omega_{h,i}}\le C\|\phi\|_{m,\Omega_i},
	\qquad m=1,2.
\end{equation}
For vector fields we use an $H^1(\operatorname{div})$ extension $E^i_{\rm div}$ such that $E^i_{\rm div}\mathbf w=\mathbf w$ on $\Omega_i$ and
\begin{equation}\label{div extension bound}
	\|E^i_{\rm div}\mathbf w\|_{2,\Omega_{h,i}}
	+\|\operatorname{div}E^i_{\rm div}\mathbf w\|_{1,\Omega_{h,i}}
	\le C\left(
	\|\mathbf w\|_{2,\Omega_i}
	+
	\|\operatorname{div}\mathbf w\|_{1,\Omega_i}
	\right).
\end{equation}
The constants in \eqref{scalar extension bound}--\eqref{div extension bound} are independent of the position of the interface relative to the mesh. Scalar Sobolev extensions are standard for smooth subdomains, and the $H^1(\operatorname{div})$ extension used here follows the construction used in unfitted mixed finite element analysis; see, for instance, \cite{CAO2022148,CaoChen2022DarcyStokes}.

Let $\Idiv$ denote the canonical $H(\operatorname{div})$ VEM interpolant, defined by matching the scaled degrees of freedom \eqref{scaled dofs}--\eqref{scaled internal dof}; see \cite{2016H(div)H(curl)} for the construction and commuting property. Let $\Pzero$ be the elementwise $L^2$ projection onto $\mathbb P_0$. The interpolants used in the error analysis are
\begin{equation}\label{early interpolation notation}
	\mathbf u_I|_K=\Idiv(E^i_{\rm div}\mathbf u_i)|_K,\qquad
	p_I|_K=\Pzero(E^ip_i)|_K,\qquad K\in\mathcal T_{h,i}^+.
\end{equation}
The commuting property of $\Idiv$ is proved in Lemma~\ref{lem commuting}.

\subsection{The discrete bilinear form}
The interface weight $\alphaG$ is the harmonic coefficient defined in \eqref{harmonic weights}.
We first introduce the correctly scaled VEM kernel stabilization. For scaled degrees of freedom define
\begin{equation}\label{D alpha definition}
	D_{\alpha,i}^K(\mathbf w,\mathbf z)
	:=\alpha_i^{-1}\sum_{r=1}^{N_K}\widehat\chi_r(\mathbf w)\widehat\chi_r(\mathbf z).
\end{equation}
The standard $L^2$-scale local VEM stabilization is
\begin{equation}\label{S alpha zero definition}
	S_{\alpha,i}^{K}(\mathbf w,\mathbf z)
	:=|K|D_{\alpha,i}^K(\mathbf w,\mathbf z)
	=\alpha_i^{-1}|K|\sum_{r=1}^{N_K}\widehat\chi_r(\mathbf w)\widehat\chi_r(\mathbf z).
\end{equation}
On interface elements we add an enhanced virtual-kernel stabilization
\begin{equation}\label{Gker definition}
	\Gker(\mathbf u_h,\mathbf v_h)
	:=\eta_{\operatorname{ker}}\sum_{i=1}^2\sum_{K\in\mathcal T_h^\Gamma}
	h_K D_{\alpha,i}^K((I-\QK)\mathbf u_{h,i},(I-\QK)\mathbf v_{h,i}),
	\qquad \eta_{\operatorname{ker}}>0.
\end{equation}
Equivalently, the kernel part on an interface element can be written with the combined coefficient
\begin{equation}\label{S alpha plus definition}
	S_{\alpha,i}^{K,+}(\mathbf w,\mathbf z)=
	\begin{cases}
		|K|D_{\alpha,i}^K(\mathbf w,\mathbf z), & K\notin\mathcal T_h^\Gamma,\\[0.4em]
		(|K|+\eta_{\operatorname{ker}}h_K)D_{\alpha,i}^K(\mathbf w,\mathbf z), & K\in\mathcal T_h^\Gamma.
	\end{cases}
\end{equation}
In the analysis below we keep $S_{\alpha,i}^K$ and $\Gker$ separated in order to display their different roles.

\begin{rmk}[Scaling of the stabilization]\label{rmk stabilization scale}
	The quantities $\widehat\chi_r(\mathbf v)$ have the size of the velocity itself. Therefore $\sum_r\widehat\chi_r(\mathbf w)^2$ is not an $L^2(K)$-scale quantity. The factor $|K|$ in \eqref{S alpha zero definition} is used so that $S_{\alpha,i}^K(\mathbf w,\mathbf w)$ is equivalent to $\alpha_i^{-1}\|\mathbf w\|_{0,K}^2$ on the kernel of $\mathbf Q_K^1$. On cut elements the additional factor $h_K$ in \eqref{Gker definition} is an extra stabilization used only to control cut element virtual kernels.
\end{rmk}

The cut-layer divergence ghost penalty is
\begin{equation}\label{Gdiv definition}
	\Gdiv(\mathbf u_h,\mathbf v_h)
	:=\eta_{\operatorname{div}}\sum_{i=1}^2\sum_{e\in\mathcal E_{\Gamma,i}}
	h_e\alpha_i^{-1}
	(\jp{\operatorname{div}\mathbf u_{h,i}},\jp{\operatorname{div}\mathbf v_{h,i}})_e,
	\qquad \eta_{\operatorname{div}}>0.
\end{equation}

We next define the interface normal-flux average. Let $\Pigam$ denote the $L^2(\Gamma_K)$-projection onto constants on each $\Gamma_K$. For $K\in\mathcal T_h^\Gamma$, define $\mathcal M_{\Gamma,K}^i(\mathbf v_h)\in\mathbb P_0(\Gamma_K)$ by
\begin{align}
	|\Gamma_K|\,\mathcal M_{\Gamma,K}^1(\mathbf v_h)
	&=\int_{K_1}\operatorname{div}\mathbf v_{h,1}\,dx
	-\int_{\partial K_1\cap\partial K}\mathbf v_{h,1}\cdot\mathbf n_{K_1}\,ds ,\label{Mgam 1}\\
	|\Gamma_K|\,\mathcal M_{\Gamma,K}^2(\mathbf v_h)
	&=\int_{\partial K_2\cap\partial K}\mathbf v_{h,2}\cdot\mathbf n_{K_2}\,ds
	-\int_{K_2}\operatorname{div}\mathbf v_{h,2}\,dx .\label{Mgam 2}
\end{align}
Here $\mathbf n_{K_i}$ is the outward unit normal to $K_i$ on $\partial K_i\cap\partial K$, and the interface normal $\mathbf n$ points from $\Omega_1$ to $\Omega_2$. Thus $\mathcal M_{\Gamma,K}^i(\mathbf v_h)$ is computable from the VEM edge moments and from the constant divergence. We set
\begin{equation}\label{Mgam jump definition}
	\jp{\Mgam\mathbf v_h}:=\mathcal M_{\Gamma,K}^1(\mathbf v_h)-\mathcal M_{\Gamma,K}^2(\mathbf v_h)
	\quad\hbox{on }\Gamma_K.
\end{equation}
The corrected interface flux used in the penalty is
\begin{equation}\label{Fgam definition}
	\mathcal F_{\Gamma,K}(\mathbf v_h)
	:=\jp{\Mgam\mathbf v_h}+(I-\Pigam)\jp{\QK\mathbf v_h\cdot\mathbf n}\quad\hbox{on }\Gamma_K.
\end{equation}
The corresponding Nitsche-type penalty is
\begin{equation}\label{PF definition}
	\PF(\mathbf u_h,\mathbf v_h)
	:=\gamma_{\mathcal F}\sum_{K\in\mathcal T_h^\Gamma}h_K^{-1}\alphaG^{-1}
	(\mathcal F_{\Gamma,K}(\mathbf u_h),\mathcal F_{\Gamma,K}(\mathbf v_h))_{\Gamma_K},
	\qquad \gamma_{\mathcal F}>0.
\end{equation}
The two components in \eqref{Fgam definition} are $L^2(\Gamma_K)$-orthogonal: the first is constant on $\Gamma_K$, while the second has integral zero on $\Gamma_K$. Thus $\PF$ controls the corrected interface flux appearing in \eqref{Fgam definition}.

Find $(\mathbf{u}_h,p_h) \in \mathbf{V}_h \times Q_h$ such that for any $(\mathbf{v}_h,q_h) \in \mathbf{V}_h \times Q_h$,
\begin{equation}\label{discrete problem}
	A_h(\mathbf{u}_h,p_h;\mathbf{v}_h,q_h) + \gamma_1 J_1(p_h,q_h)+\gamma_2J_2(p_h,q_h)
	= \sum_{i=1}^2\sum_{K\in\mathcal T_{h,i}^{+}}(f,q_{h,i})_{K_i}.
\end{equation}
Here
\begin{equation*}
	A_h(\mathbf{u}_h,p_h;\mathbf{v}_h,q_h)
	= a_h(\mathbf{u}_h,\mathbf{v}_h) + b_h(\mathbf{v}_h,p_h) - b_h(\mathbf{u}_h,q_h),
\end{equation*}
where
\begin{equation}\label{ah definition}
	\begin{aligned}
		a_h(\mathbf{u}_h,\mathbf{v}_h)
		&= \sum_{i=1}^2\sum_{K\in\mathcal T_{h,i}^{+}}
		\Bigl((\alpha_i^{-1}\QK\mathbf{u}_{h,i},\QK\mathbf{v}_{h,i})_{K_i}
		+ S_{\alpha,i}^{K,+}((I-\QK)\mathbf{u}_{h,i},(I-\QK)\mathbf{v}_{h,i})\Bigr) \\
		&\quad +\Gdiv(\mathbf u_h,\mathbf v_h)+\PF(\mathbf u_h,\mathbf v_h).
	\end{aligned}
\end{equation}
The term $\Gdiv$ supplies the cut-layer divergence extension. The enhanced stabilization $S_{\alpha,i}^{K,+}$ contains the standard $L^2$-scale VEM stabilization and, only on cut elements, the additional virtual-kernel control represented by $\Gker$.
The pressure-velocity coupling is
\begin{equation}\label{bh definition}
	b_h(\mathbf v_h,p_h)
	=\sum_{i=1}^2\sum_{K\in\mathcal T_{h,i}^{+}}
	(\operatorname{div}\mathbf v_{h,i},p_{h,i})_{K_i}
	-\sum_{K\in\mathcal T_h^\Gamma}(\jp{\Mgam\mathbf v_h},\avg{p_h})_{\Gamma_K}.
\end{equation}
Equivalently, the second term in \eqref{bh definition} may be written with $\mathcal F_{\Gamma,K}(\mathbf v_h)$, since $\avg{p_h}$ is constant on $\Gamma_K$ and the second component in \eqref{Fgam definition} has integral zero over $\Gamma_K$ and is therefore orthogonal to constants.
The ghost penalty for the scalar variable is
\begin{equation}\label{J1 definition}
	J_1(p_h,q_h)=\sum_{i=1}^2\sum_{e\in\mathcal E_{\Gamma,i}}h_e\alpha_i(\jp{p_{h,i}},\jp{q_{h,i}})_e,
\end{equation}
where $\mathcal E_{\Gamma,i}$ is the subdomain-wise transmission-edge set defined above.
The interface pressure-jump stabilization is
\begin{equation}\label{J2 definition}
	J_2(p_h,q_h)=\sum_{K\in\mathcal T_h^\Gamma}h_K\alphaG(\jp{p_h},\jp{q_h})_{\Gamma_K}.
\end{equation}

\begin{lem}[Consistency of the interface normal-flux average]\label{lem Mgam consistency}
	For every $\mathbf v_h\in\mathbf V_h$ and every $K\in\mathcal T_h^\Gamma$,
	\begin{equation}\label{Mgam exact moment}
		\int_{\Gamma_K}\mathcal M_{\Gamma,K}^i(\mathbf v_h)\,ds=\int_{\Gamma_K}\mathbf v_{h,i}\cdot\mathbf n\,ds,
		\qquad i=1,2.
	\end{equation}
	Consequently,
	\begin{equation}\label{Mgam projection identity}
		\jp{\Mgam\mathbf v_h}=\Pigam\jp{\mathbf v_h\cdot\mathbf n}
		\quad\hbox{on }\Gamma_K.
	\end{equation}
\end{lem}
\begin{proof}
	For $i=1$, the outward unit normal to $K_1$ on $\Gamma_K$ is $\mathbf n$. The divergence theorem gives
	\[
	\int_{K_1}\operatorname{div}\mathbf v_{h,1}\,dx
	=\int_{\partial K_1\cap\partial K}\mathbf v_{h,1}\cdot\mathbf n_{K_1}\,ds
	+\int_{\Gamma_K}\mathbf v_{h,1}\cdot\mathbf n\,ds,
	\]
	which is \eqref{Mgam 1}. For $i=2$, the outward normal to $K_2$ on $\Gamma_K$ is $-\mathbf n$, and the same argument gives \eqref{Mgam 2}. Since $\mathcal M_{\Gamma,K}^i$ is constant on $\Gamma_K$, \eqref{Mgam projection identity} follows from \eqref{Mgam exact moment}.
\end{proof}

\subsection{Mesh-dependent norms}
For $(\mathbf{v}_h,q_h) \in \mathbf{V}_h \times Q_h$, define
\begin{equation*}
	\begin{aligned}
		&\Vert \avg{q_h} \Vert_{-\frac{1}{2},h,\Gamma}
		= \Bigl( \sum_{K \in \mathcal{T}_h^{\Gamma}} h_K \Vert \avg{q_h} \Vert_{0,\Gamma_K}^2 \Bigr)^{\frac{1}{2}},\\
		&\Vert \phi \Vert_{\frac{1}{2},h,\Gamma}
		= \Bigl(\sum_{K \in \mathcal{T}_h^{\Gamma}} h_K^{-1} \Vert \phi \Vert_{0,\Gamma_K}^2 \Bigr)^{\frac{1}{2}}, \\
		&\NNorm{\mathbf{v}_h}^2
		= \sum_{i=1}^2\sum_{K\in\mathcal T_{h,i}^{+}}\Vert \alpha_i^{-\frac12}\QK\mathbf v_{h,i}\Vert_{0,K_i}^2
		+\sum_{i=1}^2\sum_{K\in\mathcal T_{h,i}^{+}}
		S_{\alpha,i}^{K,+}((I-\QK)\mathbf v_{h,i},(I-\QK)\mathbf v_{h,i})\\
		&\hspace{3.2cm}
		+\PF(\mathbf v_h,\mathbf v_h),\\
		&\Vert \mathbf{v}_h \Vert_h^2
		=\NNorm{\mathbf{v}_h}^2+\Vert \alpha^{-\frac12}\operatorname{div}\mathbf v_h\Vert_{0,\Omega_1\cup\Omega_2}^2
		+\Gdiv(\mathbf v_h,\mathbf v_h),\\
		&\Vert (\mathbf{v}_h,q_h) \Vert_h^2
		=\Vert \mathbf{v}_h \Vert_h^2
		+\Vert \alpha^{\frac12}q_h\Vert_{0,\Omega_1\cup\Omega_2}^2
		+J_1(q_h,q_h)+J_2(q_h,q_h).
	\end{aligned}
\end{equation*}
The norm uses the enhanced stabilization $S_{\alpha,i}^{K,+}$, which equals the standard $L^2$-scale VEM stabilization away from the interface and contains the additional kernel-component control on cut elements. The divergence ghost penalty $\Gdiv$ is included because it yields the extended-domain estimate for the discrete divergence, see Lemma~\ref{lem div ghost extension}. The corrected flux penalty $\PF$ controls the corrected interface flux \eqref{Fgam definition}.

\section{Analysis of the scheme}

Throughout the analysis, define
\begin{equation*}
	\mathcal L_h(\mathbf u_h,p_h;\mathbf v_h,q_h)
	=A_h(\mathbf u_h,p_h;\mathbf v_h,q_h)+\gamma_1J_1(p_h,q_h)+\gamma_2J_2(p_h,q_h).
\end{equation*}

\subsection{Auxiliary estimates for the unfitted VEM analysis}

\begin{lem}[Local VEM stabilization]\label{lem local S stability}
	For every element $K$ and every $\mathbf w_h\in\ker\QK$, the standard $L^2$-scaled VEM stabilization satisfies
	\begin{equation}\label{S stability}
		c_*\alpha_i^{-1}\Vert\mathbf w_h\Vert_{0,K}^2
		\le S_{\alpha,i}^{K}(\mathbf w_h,\mathbf w_h)
		\le c^*\alpha_i^{-1}\Vert\mathbf w_h\Vert_{0,K}^2 .
	\end{equation}
	The constants depend only on the shape-regularity constants of the polygonal mesh.
\end{lem}
\begin{proof}
	This is the standard finite-dimensional stability estimate for the scaled degrees of freedom in the lowest-order $H(\operatorname{div})$ VEM. It follows from the local norm equivalence on the kernel of the $L^2$ projection and the scaling in \eqref{S alpha zero definition}; see, for example, \cite{2016H(div)H(curl)}.
\end{proof}

\begin{lem}[Extended-domain estimate for piecewise constants]\label{lem ghost extension pc}
	Let $z_{h,i}$ be piecewise constant on the extended mesh $\mathcal T_{h,i}^{+}$, and let $\beta_i>0$ be constant on each subdomain. Under Assumption~\ref{ass unfitted geometry},
	\begin{equation}\label{pc extension estimate}
		\sum_{i=1}^2\Vert \beta_i^{1/2}z_{h,i}\Vert_{0,\Omega_{h,i}}^2
		\le C\sum_{i=1}^2\left(
		\Vert \beta_i^{1/2}z_{h,i}\Vert_{0,\Omega_i}^2
		+
		\sum_{e\in\mathcal E_{\Gamma,i}}h_e\beta_i\Vert\jp{z_{h,i}}\Vert_{0,e}^2
		\right).
	\end{equation}
	The constant is independent of $h$ and of the position of the interface.
\end{lem}
\begin{proof}
	This is the standard consequence of the patch condition in Assumption~\ref{ass unfitted geometry}. Along each uniformly bounded chain of transmission edges, the value of a piecewise constant on a cut element is controlled by the value on an element with a uniformly positive intersection with the corresponding subdomain plus the jumps across the edges in the chain. Summing over all chains and using the bounded overlap gives \eqref{pc extension estimate}.
\end{proof}

\begin{lem}[Trace estimate for harmonic scalar averages]\label{lem average trace}
	For every $q_h\in Q_h$,
	\begin{equation}\label{average trace estimate}
		\Vert\alphaG^{\frac12}\avg{q_h}\Vert_{-\frac12,h,\Gamma}^2
		\le C\left(\Vert\alpha^{\frac12}q_h\Vert_{0,\Omega_1\cup\Omega_2}^2+J_1(q_h,q_h)\right).
	\end{equation}
	Consequently the right-hand side can also be enlarged by $J_2(q_h,q_h)$.
\end{lem}

\begin{proof}
	By \eqref{harmonic weights},
	\begin{equation*}
		\avg{q_h}=\omega_1q_{h,1}+\omega_2q_{h,2},
		\qquad
		\alphaG\omega_1^2\le C\alpha_1,
		\qquad
		\alphaG\omega_2^2\le C\alpha_2.
	\end{equation*}
	Therefore
	\begin{align*}
		\Vert\alphaG^{1/2}\avg{q_h}\Vert_{-1/2,h,\Gamma}^2
		&=\sum_{K\in\mathcal T_h^\Gamma}h_K\Vert\alphaG^{1/2}(\omega_1q_{h,1}+\omega_2q_{h,2})\Vert_{0,\Gamma_K}^2\\
		&\le C\sum_{i=1}^2\sum_{K\in\mathcal T_h^\Gamma}h_K\alpha_i\Vert q_{h,i}\Vert_{0,\Gamma_K}^2.
	\end{align*}
	For piecewise constants, the local trace bound on an interface segment gives
	\[
	\sum_{K\in\mathcal T_h^\Gamma}h_K\alpha_i\Vert q_{h,i}\Vert_{0,\Gamma_K}^2
	\le C\Vert\alpha_i^{1/2}q_{h,i}\Vert_{0,\Omega_{h,i}}^2.
	\]
	Applying Lemma~\ref{lem ghost extension pc} with $z_{h,i}=q_{h,i}$ and $\beta_i=\alpha_i$, and then using the definition of $J_1$, gives \eqref{average trace estimate}.
\end{proof}

\begin{lem}[Commuting property of the VEM interpolant]\label{lem commuting}
	Let $\Idiv\mathbf z\in\mathbf V_h$ be the $H(\operatorname{div})$ VEM interpolant defined by the degrees of freedom \eqref{scaled dofs}--\eqref{scaled internal dof}. Then
	\begin{equation}\label{commuting property lifting}
		\operatorname{div}(\Idiv\mathbf z)=P_h^0(\operatorname{div}\mathbf z)
	\end{equation}
	on each element.
\end{lem}

\begin{proof}
	For every constant $c\in\mathbb P_0(K)$,
	\begin{equation*}
		(\operatorname{div}(\Idiv\mathbf z),c)_K
		=c\int_{\partial K}\Idiv\mathbf z\cdot\mathbf n\,ds
		=c\int_{\partial K}\mathbf z\cdot\mathbf n\,ds
		=(\operatorname{div}\mathbf z,c)_K,
	\end{equation*}
	because the edge normal-flux averages on all edges are preserved by the interpolant. Since both sides are constants on $K$, \eqref{commuting property lifting} follows.
\end{proof}

For the polygonal pressure lifting, define for $p_h\in Q_h$
\begin{equation}\label{Mh definition}
	M_h(p_h,p_h):=
	\sum_{\substack{K\notin\mathcal T_h^\Gamma\\ K\subset\Omega_i}}|K|p_{K,i}^2
	+
	\sum_{K\in\mathcal T_h^\Gamma}|K|\bar p_K^2,
	\qquad
	\bar p_K:=\frac{|K_1|p_{K,1}+|K_2|p_{K,2}}{|K|}.
\end{equation}
Here and below the first sum is understood over $i=1,2$.

In the following lifting argument, a discrete pressure \(p_h\) is also
viewed as a piecewise constant function on the whole background element.
More precisely, if \(K\in\mathcal T_{h,i}^{+}\setminus\mathcal T_h^\Gamma\),
then \(K\subset\Omega_i\) and
\[
P_K^0p_h=p_{K,i}.
\]
If \(K\in\mathcal T_h^\Gamma\), then
\[
p_h|_K=p_{K,1}\chi_{K_1}+p_{K,2}\chi_{K_2},
\]
and the whole-cell \(L^2\)-projection onto constants is
\[
P_K^0p_h
=
\frac1{|K|}\int_Kp_h\,dx
=
\frac{|K_1|p_{K,1}+|K_2|p_{K,2}}{|K|}
=
\bar p_K.
\]
Thus \(P_K^0p_h\) equals the corresponding phase value on non-cut elements
and equals the area-weighted average \(\bar p_K\) on cut elements.

\begin{lem}[Scalar extension estimate with $J_1+J_2$]\label{lem scalar patch}
	There exists a cut-independent constant $C$ such that
	\begin{equation}\label{scalar patch estimate}
		\Vert\alpha^{1/2}p_h\Vert_{0,\Omega_1\cup\Omega_2}^2
		\le
		C_\alpha\left(\alpha_{\min}M_h(p_h,p_h)+J_1(p_h,p_h)+J_2(p_h,p_h)\right).
	\end{equation}
	The constant $C_\alpha$ is independent of $h$ and of the cut position. With the present weighting, it may depend on the coefficient contrast $\alpha_{\max}/\alpha_{\min}$.
\end{lem}

\begin{proof}
	The quantity $M_h(p_h,p_h)$ contains the pressure value on non-cut elements and the area-weighted pressure average $\bar p_K$ on cut elements. It therefore does not by itself control the difference $p_{K,1}-p_{K,2}$ inside a cut element. This missing part is controlled by $J_2$, while $J_1$ gives the usual control across transmission edges associated with the same subdomain.
	
	By the local patch condition in Assumption~\ref{ass unfitted geometry}, piecewise constant values on small cut portions are controlled by their values on elements with a uniformly positive intersection with the corresponding subdomain, together with the jumps measured by $J_1$. It remains only to relate the two phase values inside each cut element. For $K\in\mathcal T_h^\Gamma$, write
	\[
	p_{K,i}=\bar p_K+(p_{K,i}-\bar p_K),
	\qquad
	\bar p_K=\frac{|K_1|p_{K,1}+|K_2|p_{K,2}}{|K|}.
	\]
	Then the average part gives the contribution $|K|\bar p_K^2$ already contained in $M_h(p_h,p_h)$, and
	\[
	\sum_{i=1}^2 |K_i|(p_{K,i}-\bar p_K)^2
	=\frac{|K_1||K_2|}{|K|}(p_{K,1}-p_{K,2})^2
	\le C h_K\Vert\jp{p_h}\Vert_{0,\Gamma_K}^2 .
	\]
	Thus the cut-element difference is controlled by $J_2$. Combining these estimates with the non-cut part of $M_h$ and multiplying by the coefficient weights gives \eqref{scalar patch estimate}. The use of $\alpha_{\min}M_h$ in this estimate is the reason why the constant may depend on the coefficient contrast.
\end{proof}

\begin{lem}[Polygonal VEM pressure lifting with $J_2$]\label{lem pressure lifting}
	For every $p_h\in Q_h$, there exists $\mathbf v_{p,h}\in\mathbf V_h$ such that
	\begin{align}
		b_h(\mathbf v_{p,h},p_h)&\ge c_p\Vert\alpha^{\frac12}p_h\Vert_{0,\Omega_1\cup\Omega_2}^2-C_p\left(J_1(p_h,p_h)+J_2(p_h,p_h)\right), \label{lifting lower bound}\\
		\Vert\mathbf v_{p,h}\Vert_h&\le C_p\left(\Vert\alpha^{\frac12}p_h\Vert_{0,\Omega_1\cup\Omega_2}^2+J_1(p_h,p_h)+J_2(p_h,p_h)\right)^{1/2}. \label{lifting norm bound}
	\end{align}
	The constants are independent of $h$ and of the cut position; in this form they may depend on the coefficient contrast.
\end{lem}

\begin{proof}
	The argument follows the standard pressure-lifting idea used in extended mixed finite element methods, adapted here to the VEM interpolant and to the reconstructed interface average flux; see, e.g., \cite{CAO2022148,CaoChen2022DarcyStokes}. Choose a continuous divergence lifting $\mathbf z_p\in H^1(\Omega)^2$ such that
	\begin{equation}\label{continuous lifting}
		\operatorname{div}\mathbf z_p=\alpha_{\min}p_h\quad\hbox{in }\Omega,
		\qquad
		\Vert\mathbf z_p\Vert_{1,\Omega}\le C\alpha_{\min}^{1/2}\Vert\alpha^{1/2}p_h\Vert_{0,\Omega_1\cup\Omega_2}.
	\end{equation}
	For instance, one may take $\mathbf z_p=\nabla\phi$, where $\phi$ solves a Poisson problem with right-hand side $\alpha_{\min}p_h$. Set
	\begin{equation*}
		\mathbf v_{p,h,1}=\mathbf v_{p,h,2}:=\Idiv\mathbf z_p
	\end{equation*}
	on every extended background element. Then $\QK\mathbf v_{p,h,1}=\QK\mathbf v_{p,h,2}$ on every cut element. Moreover, the two reconstructed average normal fluxes coincide; hence
	\begin{equation}\label{lifting jump zero}
		\jp{\Mgam\mathbf v_{p,h}}=0,
		\qquad
		\mathcal F_{\Gamma,K}(\mathbf v_{p,h})=0
		\qquad\hbox{on }\Gamma_K.
	\end{equation}
	By the commuting property,
	\[
	\operatorname{div}\mathbf v_{p,h}
	=
	P_K^0(\operatorname{div}\mathbf z_p)
	=
	\alpha_{\min}P_K^0p_h
	\quad\text{on each }K.
	\]
	Hence, on a non-cut element \(K\subset\Omega_i\),
	\[
	\operatorname{div}\mathbf v_{p,h}
	=
	\alpha_{\min}p_{K,i},
	\]
	whereas on a cut element \(K\in\mathcal T_h^\Gamma\),
	\[
	\operatorname{div}\mathbf v_{p,h}
	=
	\alpha_{\min}\bar p_K.
	\]
	Therefore, using \eqref{lifting jump zero},
	\begin{align*}
		b_h(\mathbf v_{p,h},p_h)
		&=\sum_{i=1}^2\sum_{K\in\mathcal T_{h,i}^+}(\operatorname{div}\mathbf v_{p,h,i},p_{h,i})_{K_i}\\
		&=\alpha_{\min}M_h(p_h,p_h).
	\end{align*}
	The lower bound \eqref{lifting lower bound} follows from Lemma~\ref{lem scalar patch}. The norm estimate \eqref{lifting norm bound} follows from the stability of the VEM interpolant, the local stabilization estimate \eqref{S stability}, the inverse estimate for $P_K^0(\operatorname{div}\mathbf z_p)$, and \eqref{continuous lifting}. Since $\mathcal F_{\Gamma,K}(\mathbf v_{p,h})=0$, the corrected interface penalty vanishes. The enhanced kernel stabilization is controlled by the approximation estimate
	\[
	\Vert (I-\QK)\Idiv\mathbf z_p\Vert_{0,K}\le Ch_K|\mathbf z_p|_{1,K},
	\]
	and therefore
	\[
	\Gker(\mathbf v_{p,h},\mathbf v_{p,h})
	\le C\sum_{K\in\mathcal T_h^\Gamma}h_K\alpha_i^{-1}|\mathbf z_p|_{1,K}^2
	\le C\alpha_{\min}^{-1}\Vert\mathbf z_p\Vert_{1,\Omega}^2.
	\]
	Consequently,
	\begin{equation*}
		\NNorm{\mathbf v_{p,h}}^2+
		\Vert\alpha^{-1/2}\operatorname{div}\mathbf v_{p,h}\Vert_{0,\Omega_1\cup\Omega_2}^2
		\le C\Vert\alpha^{1/2}p_h\Vert_{0,\Omega_1\cup\Omega_2}^2.
	\end{equation*}
	It remains only to estimate the added divergence ghost penalty. Since
	$\operatorname{div}\mathbf v_{p,h}=\alpha_{\min}P_K^0p_h$, its jump across a transmission edge is the jump of the corresponding cell averages of $p_h$, up to the factor $\alpha_{\min}$. The same patch argument used in Lemma~\ref{lem scalar patch} therefore yields
	\begin{equation*}
		\Gdiv(\mathbf v_{p,h},\mathbf v_{p,h})
		\le C\left(\Vert\alpha^{1/2}p_h\Vert_{0,\Omega_1\cup\Omega_2}^2+J_1(p_h,p_h)+J_2(p_h,p_h)\right).
	\end{equation*}
	Here $J_1$ controls jumps between neighboring extended elements associated with the same subdomain, while $J_2$ controls the two values inside cut elements. Combining the last two estimates proves \eqref{lifting norm bound}.
\end{proof}

\begin{lem}[Divergence ghost penalty extension]
	\label{lem div ghost extension}
	For every $\mathbf v_h\in\mathbf V_h$,
	\begin{equation}\label{active divergence extension}
		\Vert\alpha^{-1/2}\operatorname{div}\mathbf v_h\Vert_{0,\Omega_{h,1}\cup\Omega_{h,2}}^2
		\le C\left(
		\Vert\alpha^{-1/2}\operatorname{div}\mathbf v_h\Vert_{0,\Omega_1\cup\Omega_2}^2
		+\Gdiv(\mathbf v_h,\mathbf v_h)
		\right).
	\end{equation}
	Consequently, for $r_h=\alpha^{-1}\operatorname{div}\mathbf v_h$,
	\begin{equation}\label{div ghost bridge closed}
		J_1(r_h,r_h)+J_2(r_h,r_h)
		\le C\left(
		\Vert\alpha^{-1/2}\operatorname{div}\mathbf v_h\Vert_{0,\Omega_1\cup\Omega_2}^2
		+\Gdiv(\mathbf v_h,\mathbf v_h)
		\right)
		\le C\Vert\mathbf v_h\Vert_h^2 .
	\end{equation}
\end{lem}

\begin{proof}
	Let $d_{h,i}=\operatorname{div}\mathbf v_{h,i}$. Since $d_{h,i}$ is constant on every extended background element, \eqref{active divergence extension} is exactly Lemma~\ref{lem ghost extension pc} with $z_{h,i}=d_{h,i}$ and $\beta_i=\alpha_i^{-1}$:
	\begin{align*}
		\Vert\alpha_i^{-1/2}d_{h,i}\Vert_{0,\Omega_{h,i}}^2
		&\le C\left(
		\Vert\alpha_i^{-1/2}d_{h,i}\Vert_{0,\Omega_i}^2
		+\sum_{e\in\mathcal E_{\Gamma,i}}h_e\alpha_i^{-1}\Vert\jp{d_{h,i}}\Vert_{0,e}^2
		\right)  \\
		&\le C\left(
		\Vert\alpha_i^{-1/2}d_{h,i}\Vert_{0,\Omega_i}^2
		+\eta_{\operatorname{div}}^{-1}G_{{\operatorname{div}},i}(\mathbf v_h,\mathbf v_h)
		\right).
	\end{align*}
	Summing over $i=1,2$ proves \eqref{active divergence extension}.
	
	For $J_1$ we have directly
	\begin{align*}
		J_1(r_h,r_h)
		&=\sum_{i=1}^2\sum_{e\in\mathcal E_{\Gamma,i}}h_e\alpha_i\Vert\jp{\alpha_i^{-1}\operatorname{div}\mathbf v_{h,i}}\Vert_{0,e}^2  \\
		&=\sum_{i=1}^2\sum_{e\in\mathcal E_{\Gamma,i}}h_e\alpha_i^{-1}\Vert\jp{\operatorname{div}\mathbf v_{h,i}}\Vert_{0,e}^2
		\le C\Gdiv(\mathbf v_h,\mathbf v_h).
	\end{align*}
	For $J_2$, using $\alpha_\Gamma\alpha_i^{-2}\le C\alpha_i^{-1}$ and the fact that $r_{h,i}$ is constant on $K$,
	\begin{align*}
		J_2(r_h,r_h)
		&\le C\sum_{K\in\mathcal T_h^\Gamma}\sum_{i=1}^2 h_K\alpha_\Gamma\Vert\alpha_i^{-1}\operatorname{div}\mathbf v_{h,i}\Vert_{0,\Gamma_K}^2 \\
		&\le C\sum_{K\in\mathcal T_h^\Gamma}\sum_{i=1}^2\Vert\alpha_i^{-1/2}\operatorname{div}\mathbf v_{h,i}\Vert_{0,K}^2 \\
		&\le C\Vert\alpha^{-1/2}\operatorname{div}\mathbf v_h\Vert_{0,\Omega_{h,1}\cup\Omega_{h,2}}^2.
	\end{align*}
	Combining this with \eqref{active divergence extension} completes the proof.
\end{proof}

\begin{lem}[Cut-cell virtual-kernel and inverse trace control]\label{lem Gker kernel control}
	Let $\mathbf w_h=(I-\QK)\mathbf v_h$ on each cut element. Then
	\begin{equation}\label{Gker L2 control}
		\sum_{i=1}^2\sum_{K\in\mathcal T_h^\Gamma}
		\Vert\alpha_i^{-1/2}\mathbf w_{h,i}\Vert_{0,K}^2
		\le C h\,\Gker(\mathbf v_h,\mathbf v_h).
	\end{equation}
	Moreover, the following virtual-kernel inverse trace inequality
	\begin{equation}\label{kernel trace inequality}
		\Vert \mathbf w_{h,i}\cdot\mathbf n\Vert_{0,\Gamma_K}^2
		\le C h_K^{-1}\Vert\mathbf w_{h,i}\Vert_{0,K}^2,
	\end{equation}
	holds, and consequently
	\begin{equation}\label{Gker trace control}
		\sum_{K\in\mathcal T_h^\Gamma}h_K^{-1}\alphaG^{-1}
		\Vert (I-\Pigam)\jp{(I-\QK)\mathbf v_h\cdot\mathbf n}\Vert_{0,\Gamma_K}^2
		\le C_\alpha h^{-1}\Gker(\mathbf v_h,\mathbf v_h).
	\end{equation}
	The constant $C_\alpha$ is independent of $h$ and of the cut position; in this weighting it may depend on the coefficient contrast.
\end{lem}
\begin{proof}
	By Lemma~\ref{lem local S stability} and the definition \eqref{Gker definition},
	\[
	\alpha_i^{-1}\Vert\mathbf w_{h,i}\Vert_{0,K}^2
	\le C |K|D_{\alpha,i}^K(\mathbf w_{h,i},\mathbf w_{h,i})
	\le C h_K\,h_KD_{\alpha,i}^K(\mathbf w_{h,i},\mathbf w_{h,i}).
	\]
	After summing over cut elements this gives \eqref{Gker L2 control}. The estimate \eqref{kernel trace inequality} is an inverse trace inequality on the finite-dimensional VEM kernel. Equivalently, map $K$ to a reference polygon. On the reference element, all norms on the finite-dimensional kernel are equivalent, and the normal trace on any admissible internal interface segment is a bounded linear functional with respect to the reference $L^2$ norm. Scaling back to $K$ gives the factor $h_K^{-1}$. Then \eqref{Gker trace control} follows from the projection stability of $I-\Pigam$, the inverse trace bound \eqref{kernel trace inequality}, $\alphaG^{-1}\le C_\alpha\alpha_i^{-1}$, and the local form of \eqref{Gker L2 control}.
\end{proof}

\begin{lem}[Interface projection consistency]\label{lem interface projection consistency}
	Let $p$ be continuous across $\Gamma$ and let $\Pigam p$ be the piecewise constant projection on each $\Gamma_K$. Assume $\mathbf u=\alpha\nabla p$ and $\mathbf u\in H^2(\Omega_1\cup\Omega_2)^2$. Then, for all $\mathbf v_h\in\mathbf V_h$,
	\begin{equation}\label{interface projection consistency estimate}
		\left|
		\sum_{K\in\mathcal T_h^\Gamma}
		\bigl((I-\Pigam)p,\,
		\jp{\mathbf v_h\cdot\mathbf n}-\jp{\Mgam\mathbf v_h}\bigr)_{\Gamma_K}
		\right|
		\le C_\alpha h\Vert\alpha^{-1/2}\mathbf u\Vert_{2,\Omega_1\cup\Omega_2}\Vert\mathbf v_h\Vert_h.
	\end{equation}
\end{lem}
\begin{proof}
	Set $V_K(\mathbf v_h)=\jp{\mathbf v_h\cdot\mathbf n}|_{\Gamma_K}$. By Lemma~\ref{lem Mgam consistency},
	\[
	\jp{\Mgam\mathbf v_h}=\Pigam V_K(\mathbf v_h),
	\]
	so the left-hand side in \eqref{interface projection consistency estimate} equals
	\[
	\left|\sum_{K\in\mathcal T_h^\Gamma}((I-\Pigam)p,(I-\Pigam)V_K(\mathbf v_h))_{\Gamma_K}\right|.
	\]
	By weighted Cauchy's inequality this is bounded by $A_pB_v$, where
	\[
	A_p^2=\sum_{K\in\mathcal T_h^\Gamma}h_K\alphaG\Vert(I-\Pigam)p\Vert_{0,\Gamma_K}^2,
	\]
	and
	\[
	B_v^2=\sum_{K\in\mathcal T_h^\Gamma}h_K^{-1}\alphaG^{-1}\Vert(I-\Pigam)V_K(\mathbf v_h)\Vert_{0,\Gamma_K}^2.
	\]
	For \(A_p\), the one-dimensional Poincaré inequality on \(\Gamma_K\)
	gives
	\[
	\|(I-\Pi_{\Gamma,K}^0)p\|_{0,\Gamma_K}
	\le
	Ch_K\|\partial_\tau p\|_{0,\Gamma_K}.
	\]
	Since \(p\) is continuous across \(\Gamma\) and
	\(\mathbf u=\alpha\nabla p\), we have on each side of the interface
	\[
	\partial_\tau p_i=\alpha_i^{-1}\mathbf u_i\cdot\tau .
	\]
	Moreover, \(\alpha_\Gamma\alpha_i^{-2}\le C\alpha_i^{-1}\). Hence,
	by the trace theorem on the smooth interface,
	\[
	\begin{aligned}
		A_p^2
		&\le
		C\sum_{K\in\mathcal T_h^\Gamma}
		h_K^3\alpha_\Gamma
		\|\partial_\tau p\|_{0,\Gamma_K}^2 \le
		C h^3
		\|\alpha^{-1/2}\mathbf u\|_{0,\Gamma}^2  \\
		&\le
		C h^3
		\|\alpha^{-1/2}\mathbf u\|_{1,\Omega_1\cup\Omega_2}^2
		\le
		C h^3
		\|\alpha^{-1/2}\mathbf u\|_{2,\Omega_1\cup\Omega_2}^2 .
	\end{aligned}
	\]
	Therefore,
	\[
	A_p\le
	Ch^{3/2}
	\|\alpha^{-1/2}\mathbf u\|_{2,\Omega_1\cup\Omega_2}.
	\]
	For $B_v$, decompose $\mathbf v_h=\QK\mathbf v_h+(I-\QK)\mathbf v_h$. The polynomial component is controlled by the corrected flux penalty because
	\[
	(I-\Pigam)\jp{\QK\mathbf v_h\cdot\mathbf n}
	\]
	is one component of $\mathcal F_{\Gamma,K}(\mathbf v_h)$. Hence
	\[
	\sum_Kh_K^{-1}\alphaG^{-1}\Vert(I-\Pigam)\jp{\QK\mathbf v_h\cdot\mathbf n}\Vert_{0,\Gamma_K}^2
	\le C\PF(\mathbf v_h,\mathbf v_h).
	\]
	The kernel component is controlled by Lemma~\ref{lem Gker kernel control}, namely by \eqref{Gker trace control}. Therefore
	\[
	B_v\le C_\alpha h^{-1/2}\Vert\mathbf v_h\Vert_h.
	\]
	Combining the estimates for $A_p$ and $B_v$ gives \eqref{interface projection consistency estimate}.
\end{proof}

\begin{lem}[Interpolation estimates for the new cut penalties]\label{lem interpolation new penalties}
	Let $\mathbf u_I=\Idiv E^i_{\rm div}\mathbf u_i$ on each extended mesh. If $\mathbf u\in H^2(\Omega_1\cup\Omega_2)^2$ and $\operatorname{div}\mathbf u\in H^1(\Omega_1\cup\Omega_2)$, then
	\begin{align}
		\Gker(\mathbf u-\mathbf u_I,\mathbf u-\mathbf u_I)
		&\le Ch^2\Vert\alpha^{-1/2}\mathbf u\Vert_{2,\Omega_1\cup\Omega_2}^2,\label{interpolation Gker term}\\
		\PF(\mathbf u-\mathbf u_I,\mathbf u-\mathbf u_I)
		&\le Ch^2\left(\Vert\alpha^{-1/2}\mathbf u\Vert_{2,\Omega_1\cup\Omega_2}^2
		+\Vert\alpha^{-1/2}\operatorname{div}\mathbf u\Vert_{1,\Omega_1\cup\Omega_2}^2\right).
		\label{interpolation PF term}
	\end{align}
\end{lem}
\begin{proof}
	For \(K\in\mathcal T_{h,i}^{+}\), set
	\[
	\mathbf w_i
	:=
	E^i_{\rm div}\mathbf u_i
	-
	\mathcal I_h^{\rm div}(E^i_{\rm div}\mathbf u_i).
	\]
	Then \(\mathbf w_i=\mathbf u_i-\mathbf u_{I,i}\) on \(K_i=K\cap\Omega_i\).
	
	We first prove \((4.20)\). On cut elements, the enhanced kernel part is
	equivalent to \(h_K^{-1}\) times the usual \(L^2\)-scaled VEM
	stabilization on the kernel. Hence, by the local finite-dimensional
	equivalence of the scaled degrees of freedom,
	\[
	G_{\ker}(\mathbf u-\mathbf u_I,\mathbf u-\mathbf u_I)
	\le
	C
	\sum_{i=1}^2
	\sum_{K\in\mathcal T_h^\Gamma}
	h_K^{-1}
	\|\alpha_i^{-1/2}(I-\mathbf Q_K^1)\mathbf w_i\|_{0,K}^2 .
	\]
	Using the stability of \(I-\mathbf Q_K^1\) and the standard
	\(H(\operatorname{div})\)-VEM interpolation estimate, we have
	\[
	\|(I-\mathbf Q_K^1)\mathbf w_i\|_{0,K}
	\le
	C\|\mathbf w_i\|_{0,K}
	\le
	C h_K^2 |E^i_{\rm div}\mathbf u_i|_{2,K}.
	\]
	Therefore,
	\[
	\begin{aligned}
		G_{\ker}(\mathbf u-\mathbf u_I,\mathbf u-\mathbf u_I)
		&\le
		C
		\sum_{i=1}^2
		\sum_{K\in\mathcal T_h^\Gamma}
		h_K^{-1}\alpha_i^{-1}
		h_K^4 |E^i_{\rm div}\mathbf u_i|_{2,K}^2   \\
		&\le
		C h^2
		\sum_{i=1}^2
		\|\alpha_i^{-1/2}E^i_{\rm div}\mathbf u_i\|_{2,\Omega_{h,i}}^2 .
	\end{aligned}
	\]
	By the extension estimate,
	\[
	G_{\ker}(\mathbf u-\mathbf u_I,\mathbf u-\mathbf u_I)
	\le
	C h^2
	\|\alpha^{-1/2}\mathbf u\|_{2,\Omega_1\cup\Omega_2}^2 .
	\]
	This proves \((4.20)\).
	
	We now prove \((4.21)\). By the definition of the corrected flux,
	\[
	\mathcal F_{\Gamma,K}(\mathbf u-\mathbf u_I)
	=
	[\![\mathcal M_\Gamma(\mathbf u-\mathbf u_I)]\!]
	+
	(I-\Pi_{\Gamma,K}^0)
	[\![\mathbf Q_K^1(\mathbf u-\mathbf u_I)\cdot\mathbf n]\!].
	\]
	The first term is constant on \(\Gamma_K\), while the second one has
	zero average on \(\Gamma_K\). Hence they are \(L^2(\Gamma_K)\)-orthogonal,
	and it is enough to estimate the two parts separately.
	
	We begin with the reconstructed average-flux part. Since
	\(\mathcal M_{\Gamma,K}^i(\mathbf w_i)\in\mathbb P_0(\Gamma_K)\), we have
	\[
	\|\mathcal M_{\Gamma,K}^i(\mathbf w_i)\|_{0,\Gamma_K}^2
	=
	|\Gamma_K|^{-1}
	\left|
	\int_{\Gamma_K}
	\mathcal M_{\Gamma,K}^i(\mathbf w_i)\,ds
	\right|^2 .
	\]
	We use the following local estimate for this average-flux interpolation error:
	\begin{equation}\label{local M interpolation estimate}
		\left|
		\int_{\Gamma_K}\mathcal M_{\Gamma,K}^i(\mathbf w_i)\,ds
		\right|
		\le
		C h_K^{3/2}|\Gamma_K|^{1/2}
		\left(
		|E^i_{\rm div}\mathbf u_i|_{2,K}
		+
		\|\operatorname{div}E^i_{\rm div}\mathbf u_i\|_{1,K}
		\right).
	\end{equation}
	Let us briefly justify \eqref{local M interpolation estimate}. By the definition of
	\(\mathcal M_{\Gamma,K}^i\) and the divergence theorem on \(K_i\),
	\[
	\int_{\Gamma_K}\mathcal M_{\Gamma,K}^i(\mathbf w_i)\,ds
	=
	\sigma_i
	\left(
	\int_{K_i}\operatorname{div}\mathbf w_i\,dx
	-
	\int_{\partial K_i\cap\partial K}
	\mathbf w_i\cdot\mathbf n_{K_i}\,ds
	\right),
	\]
	where \(\sigma_i=\pm1\) depends only on the orientation of the interface normal. The volume term is bounded by the commuting property and the \(P_K^0\)-projection estimate. The boundary term is bounded by using the normal-flux moment preservation of the VEM interpolant on full background edges, so that only the portions cut off by the interface remain, followed by the trace estimate on the cut portions of the background edges. Under the graph assumption on \(\Gamma_K\), these estimates give exactly \eqref{local M interpolation estimate}, with constants independent of the cut position.
	Consequently,
	\[
	\|\mathcal M_{\Gamma,K}^i(\mathbf w_i)\|_{0,\Gamma_K}^2
	\le
	C h_K^3
	\left(
	|E^i_{\rm div}\mathbf u_i|_{2,K}^2
	+
	\|\operatorname{div}E^i_{\rm div}\mathbf u_i\|_{1,K}^2
	\right).
	\]
	Hence,
	\[
	\begin{aligned}
		&\sum_{K\in\mathcal T_h^\Gamma}
		h_K^{-1}\alpha_\Gamma^{-1}
		\|[\![\mathcal M_\Gamma(\mathbf u-\mathbf u_I)]\!]\|_{0,\Gamma_K}^2 \\
		&\qquad
		\le
		C
		\sum_{i=1}^2
		\sum_{K\in\mathcal T_h^\Gamma}
		h_K^2\alpha_i^{-1}
		\left(
		|E^i_{\rm div}\mathbf u_i|_{2,K}^2
		+
		\|\operatorname{div}E^i_{\rm div}\mathbf u_i\|_{1,K}^2
		\right)                                      \\
		&\qquad
		\le
		C h^2
		\left(
		\|\alpha^{-1/2}\mathbf u\|_{2,\Omega_1\cup\Omega_2}^2
		+
		\|\alpha^{-1/2}\operatorname{div}\mathbf u\|_{1,\Omega_1\cup\Omega_2}^2
		\right),
	\end{aligned}
	\]
	where we used the extension estimates. Here and below the constant may
	depend on the coefficient contrast.
	
	It remains to estimate the zero-mean projected-flux part. Since
	\(I-\Pi_{\Gamma,K}^0\) is an \(L^2(\Gamma_K)\)-orthogonal projection,
	\[
	\begin{aligned}
		&
		\|(I-\Pi_{\Gamma,K}^0)
		[\![\mathbf Q_K^1(\mathbf u-\mathbf u_I)\cdot\mathbf n]\!]\|_{0,\Gamma_K}
		\\
		&\qquad
		\le
		\|[\![\mathbf Q_K^1(\mathbf u-\mathbf u_I)\cdot\mathbf n]\!]\|_{0,\Gamma_K}.
	\end{aligned}
	\]
	For each side \(i=1,2\), the inverse trace inequality for the polynomial
	\(\mathbf Q_K^1\mathbf w_i\) gives
	\[
	\|\mathbf Q_K^1\mathbf w_i\cdot\mathbf n\|_{0,\Gamma_K}^2
	\le
	C h_K^{-1}
	\|\mathbf Q_K^1\mathbf w_i\|_{0,K}^2 .
	\]
	By the \(L^2(K)\)-stability of \(\mathbf Q_K^1\) and the interpolation
	estimate,
	\[
	\|\mathbf Q_K^1\mathbf w_i\|_{0,K}
	\le
	\|\mathbf w_i\|_{0,K}
	\le
	C h_K^2 |E^i_{\rm div}\mathbf u_i|_{2,K}.
	\]
	Thus,
	\[
	\|\mathbf Q_K^1\mathbf w_i\cdot\mathbf n\|_{0,\Gamma_K}^2
	\le
	C h_K^3 |E^i_{\rm div}\mathbf u_i|_{2,K}^2 .
	\]
	Consequently,
	\[
	\begin{aligned}
		&\sum_{K\in\mathcal T_h^\Gamma}
		h_K^{-1}\alpha_\Gamma^{-1}
		\|(I-\Pi_{\Gamma,K}^0)
		[\![\mathbf Q_K^1(\mathbf u-\mathbf u_I)\cdot\mathbf n]\!]\|_{0,\Gamma_K}^2
		\\
		&\qquad
		\le
		C
		\sum_{i=1}^2
		\sum_{K\in\mathcal T_h^\Gamma}
		h_K^2\alpha_i^{-1}
		|E^i_{\rm div}\mathbf u_i|_{2,K}^2                    \\
		&\qquad
		\le
		C h^2
		\|\alpha^{-1/2}\mathbf u\|_{2,\Omega_1\cup\Omega_2}^2 .
	\end{aligned}
	\]
	Combining the two bounds and using the definition of \(P_{\mathcal F}\),
	we obtain
	\[
	P_{\mathcal F}(\mathbf u-\mathbf u_I,\mathbf u-\mathbf u_I)
	\le
	C h^2
	\left(
	\|\alpha^{-1/2}\mathbf u\|_{2,\Omega_1\cup\Omega_2}^2
	+
	\|\alpha^{-1/2}\operatorname{div}\mathbf u\|_{1,\Omega_1\cup\Omega_2}^2
	\right).
	\]
	This proves \((4.21)\), and the lemma follows.
\end{proof}

\subsection{Continuity analysis}
\begin{lem}\label{Continuity}
	For all $(\mathbf{u}_h,p_h)$, $(\mathbf{v}_h,q_h) \in \mathbf{V}_h \times Q_h$, there holds
	\begin{equation}\label{boundness}
		|\mathcal{L}_h(\mathbf{u}_h,p_h;\mathbf{v}_h,q_h)|
		\leq C_b \Vert (\mathbf{u}_h,p_h) \Vert_h \Vert (\mathbf{v}_h,q_h) \Vert_h.
	\end{equation}
\end{lem}
\begin{proof}
	All terms in $a_h$ are components of the velocity norm, and hence
	\[
	|a_h(\mathbf u_h,\mathbf v_h)|
	\le C\Vert\mathbf u_h\Vert_h\Vert\mathbf v_h\Vert_h.
	\]
	For the coupling term, Cauchy's inequality, the pressure-average trace estimate \eqref{average trace estimate}, and the fact that $\PF$ controls $\jp{\Mgam\mathbf u_h}$ give
	\[
	\begin{aligned}
		|b_h(\mathbf u_h,q_h)|
		&\le
		\Vert\alpha^{-1/2}\operatorname{div}\mathbf u_h\Vert_{0,\Omega_1\cup\Omega_2}
		\Vert\alpha^{1/2}q_h\Vert_{0,\Omega_1\cup\Omega_2} \\
		&\quad+
		\Vert\alphaG^{-1/2}\jp{\Mgam\mathbf u_h}\Vert_{1/2,h,\Gamma}
		\Vert\alphaG^{1/2}\avg{q_h}\Vert_{-1/2,h,\Gamma} \\
		&\le C\Vert\mathbf u_h\Vert_h\Vert(0,q_h)\Vert_h.
	\end{aligned}
	\]
	The same argument gives
	\[
	|b_h(\mathbf v_h,p_h)|
	\le C\Vert\mathbf v_h\Vert_h\Vert(0,p_h)\Vert_h.
	\]
	The pressure stabilization terms are bounded by Cauchy's inequality. Combining these bounds proves \eqref{boundness}.
\end{proof}

\subsection{Discrete inf-sup stability}
\begin{thm}[Discrete inf-sup stability]\label{thm inf sup}
	For sufficiently small $h$, there exists a positive constant $C_s$ independent of $h$ and of the cut position such that, for all $(\mathbf u_h,p_h)\in\mathbf V_h\times Q_h$,
	\begin{equation}\label{Coercivity}
		\sup_{(\mathbf v_h,q_h)\in\mathbf V_h\times Q_h}
		\frac{\mathcal L_h(\mathbf u_h,p_h;\mathbf v_h,q_h)}{\Vert(\mathbf v_h,q_h)\Vert_h}
		\ge C_s\Vert(\mathbf u_h,p_h)\Vert_h.
	\end{equation}
\end{thm}
\begin{proof}
	Let \(\mathbf v_{p,h}\) be the pressure-lifting function given by
	Lemma~4.8. We use the test pair
	\[
	(\mathbf v_h,q_h)
	=
	\left(
	\mathbf u_h+\delta_1\mathbf v_{p,h},
	\;
	p_h-\delta_2\alpha^{-1}\operatorname{div}\mathbf u_h
	\right),
	\qquad
	\delta_1,\delta_2>0,
	\]
	where \(\delta_1,\delta_2\) will be chosen later. Set
	\[
	r_h:=\alpha^{-1}\operatorname{div}\mathbf u_h.
	\]
	For brevity, denote
	\[
	\mathcal U_h:=\vert\!\vert\!\vert\mathbf u_h\vert\!\vert\!\vert^2,
	\qquad
	\mathcal G_h:=G_{\rm div}(\mathbf u_h,\mathbf u_h),
	\]
	\[
	\mathcal D_h:=
	\|\alpha^{-1/2}\operatorname{div}\mathbf u_h\|_{0,\Omega_1\cup\Omega_2}^2,
	\qquad
	\mathcal P_h:=
	\|\alpha^{1/2}p_h\|_{0,\Omega_1\cup\Omega_2}^2,
	\]
	and
	\[
	\mathcal J_{1,h}:=J_1(p_h,p_h),
	\qquad
	\mathcal J_{2,h}:=J_2(p_h,p_h).
	\]
	
	By bilinearity,
	\begin{equation}\label{eq:infsup-splitting}
		\begin{aligned}
			\mathcal L_h(\mathbf u_h,p_h;\mathbf v_h,q_h)
			&=
			\mathcal L_h(\mathbf u_h,p_h;\mathbf u_h,p_h)
			+\delta_1\mathcal L_h(\mathbf u_h,p_h;\mathbf v_{p,h},0)  \\
			&\quad
			-\delta_2\mathcal L_h(\mathbf u_h,p_h;0,r_h).
		\end{aligned}
	\end{equation}
	
	We first estimate the three terms on the right-hand side. Since the two
	coupling terms cancel when \((\mathbf v_h,q_h)=(\mathbf u_h,p_h)\), we have
	\begin{equation}\label{eq:infsup-basic-term}
		\begin{aligned}
			\mathcal L_h(\mathbf u_h,p_h;\mathbf u_h,p_h)
			&=
			a_h(\mathbf u_h,\mathbf u_h)
			+\gamma_1J_1(p_h,p_h)+\gamma_2J_2(p_h,p_h)      \\
			&\ge
			\mathcal U_h+\mathcal G_h
			+\gamma_1\mathcal J_{1,h}
			+\gamma_2\mathcal J_{2,h}.
		\end{aligned}
	\end{equation}
	
	Next we estimate the pressure-lifting term. By the definition of
	\(\mathcal L_h\),
	\[
	\mathcal L_h(\mathbf u_h,p_h;\mathbf v_{p,h},0)
	=
	a_h(\mathbf u_h,\mathbf v_{p,h})
	+
	b_h(\mathbf v_{p,h},p_h).
	\]
	Lemma~4.8 gives
	\[
	b_h(\mathbf v_{p,h},p_h)
	\ge
	c_p\mathcal P_h
	-
	C_p(\mathcal J_{1,h}+\mathcal J_{2,h}),
	\]
	and the continuity of \(a_h\), together with the stability estimate for
	\(\mathbf v_{p,h}\) in Lemma~4.8, gives
	\[
	|a_h(\mathbf u_h,\mathbf v_{p,h})|
	\le
	C(\mathcal U_h+\mathcal G_h)^{1/2}
	(\mathcal P_h+\mathcal J_{1,h}+\mathcal J_{2,h})^{1/2}.
	\]
	Hence, by Young's inequality, for any \(\varepsilon_1>0\),
	\begin{equation}\label{eq:infsup-pressure-lifting}
		\begin{aligned}
			\mathcal L_h(\mathbf u_h,p_h;\mathbf v_{p,h},0)
			&\ge
			(c_p-C\varepsilon_1)\mathcal P_h
			-(C_p+C\varepsilon_1)(\mathcal J_{1,h}+\mathcal J_{2,h})  \\
			&\quad
			-\frac{C}{\varepsilon_1}(\mathcal U_h+\mathcal G_h).
		\end{aligned}
	\end{equation}
	
	It remains to estimate the divergence test. Since
	\[
	\mathcal L_h(\mathbf u_h,p_h;0,r_h)
	=
	-b_h(\mathbf u_h,r_h)
	+\gamma_1J_1(p_h,r_h)
	+\gamma_2J_2(p_h,r_h),
	\]
	we have
	\[
	-\mathcal L_h(\mathbf u_h,p_h;0,r_h)
	=
	b_h(\mathbf u_h,r_h)
	-\gamma_1J_1(p_h,r_h)
	-\gamma_2J_2(p_h,r_h).
	\]
	The volume part of \(b_h(\mathbf u_h,r_h)\) gives exactly
	\[
	\sum_{i=1}^2\sum_{K}
	(\operatorname{div}\mathbf u_{h,i},r_{h,i})_{K_i}
	=
	\mathcal D_h.
	\]
	The interface part is bounded by the trace estimate for the harmonic
	average and Lemma~4.9:
	\[
	\left|
	\sum_{K\in\mathcal T_h^\Gamma}
	([\![\mathcal M_\Gamma\mathbf u_h]\!],\{r_h\})_{\Gamma_K}
	\right|
	\le
	C\mathcal U_h^{1/2}
	\left(\mathcal D_h+\mathcal G_h\right)^{1/2}.
	\]
	Therefore, for any \(\varepsilon_2>0\),
	\[
	b_h(\mathbf u_h,r_h)
	\ge
	(1-C\varepsilon_2)\mathcal D_h
	-\frac{C}{\varepsilon_2}\mathcal U_h
	-C\varepsilon_2\mathcal G_h.
	\]
	
	For the stabilization cross terms, Cauchy's inequality and Young's
	inequality give, for any \(\varepsilon_3,\varepsilon_4>0\),
	\[
	\gamma_1|J_1(p_h,r_h)|
	\le
	\gamma_1\frac{C}{\varepsilon_3}\mathcal J_{1,h}
	+
	C\gamma_1\varepsilon_3 J_1(r_h,r_h),
	\]
	and
	\[
	\gamma_2|J_2(p_h,r_h)|
	\le
	\gamma_2\frac{C}{\varepsilon_4}\mathcal J_{2,h}
	+
	C\gamma_2\varepsilon_4 J_2(r_h,r_h).
	\]
	By Lemma~4.9,
	\[
	J_1(r_h,r_h)+J_2(r_h,r_h)
	\le
	C(\mathcal D_h+\mathcal G_h).
	\]
	Combining the preceding estimates yields
	\begin{equation}\label{eq:infsup-divergence-test}
		\begin{aligned}
			-\mathcal L_h(\mathbf u_h,p_h;0,r_h)
			&\ge
			(1-C\varepsilon_2-C\gamma_1\varepsilon_3-C\gamma_2\varepsilon_4)
			\mathcal D_h                                      \\
			&\quad
			-\frac{C}{\varepsilon_2}\mathcal U_h
			-C(\varepsilon_2+\gamma_1\varepsilon_3+\gamma_2\varepsilon_4)
			\mathcal G_h                                      \\
			&\quad
			-\gamma_1\frac{C}{\varepsilon_3}\mathcal J_{1,h}
			-\gamma_2\frac{C}{\varepsilon_4}\mathcal J_{2,h}.
		\end{aligned}
	\end{equation}
	This is the point where the divergence ghost penalty is used: Lemma~4.9
	does not bound \(J_1(r_h,r_h)+J_2(r_h,r_h)\) by the physical divergence
	norm alone, but by the physical divergence norm together with
	\(G_{\rm div}(\mathbf u_h,\mathbf u_h)\).
	
	Now we insert
	\eqref{eq:infsup-basic-term}, \eqref{eq:infsup-pressure-lifting}, and
	\eqref{eq:infsup-divergence-test} into \eqref{eq:infsup-splitting}. We obtain
	\begin{equation}\label{eq:infsup-combined-lower}
		\begin{aligned}
			\mathcal L_h(\mathbf u_h,p_h;\mathbf v_h,q_h)
			&\ge
			\left(1-\delta_1\frac{C}{\varepsilon_1}
			-\delta_2\frac{C}{\varepsilon_2}\right)\mathcal U_h       \\
			&\quad
			+\left(1-\delta_1\frac{C}{\varepsilon_1}
			-\delta_2C(\varepsilon_2+\gamma_1\varepsilon_3+\gamma_2\varepsilon_4)
			\right)\mathcal G_h                                             \\
			&\quad
			+\delta_1(c_p-C\varepsilon_1)\mathcal P_h                         \\
			&\quad
			+\delta_2(1-C\varepsilon_2-C\gamma_1\varepsilon_3-C\gamma_2\varepsilon_4)
			\mathcal D_h                                                      \\
			&\quad
			+\left(\gamma_1-\delta_1(C_p+C\varepsilon_1)
			-\delta_2\gamma_1\frac{C}{\varepsilon_3}\right)
			\mathcal J_{1,h}                                                  \\
			&\quad
			+\left(\gamma_2-\delta_1(C_p+C\varepsilon_1)
			-\delta_2\gamma_2\frac{C}{\varepsilon_4}\right)
			\mathcal J_{2,h}.
		\end{aligned}
	\end{equation}
	
	We now choose the parameters. First choose
	\(\varepsilon_1>0\) sufficiently small so that
	\(c_p-C\varepsilon_1>0\). Then choose
	\(\varepsilon_2,\varepsilon_3,\varepsilon_4>0\) sufficiently small so that
	\[
	1-C\varepsilon_2-C\gamma_1\varepsilon_3-C\gamma_2\varepsilon_4>0.
	\]
	After these Young parameters are fixed, choose
	\(\delta_1,\delta_2>0\) sufficiently small so that all coefficients in
	\eqref{eq:infsup-combined-lower} are positive. Hence there exists a
	constant \(C_I>0\), independent of \(h\) and of the cut position, such that
	\begin{equation}\label{eq:infsup-positive-test}
		\mathcal L_h(\mathbf u_h,p_h;\mathbf v_h,q_h)
		\ge
		C_I\|(\mathbf u_h,p_h)\|_h^2 .
	\end{equation}
	
	It remains to bound the norm of the chosen test pair. By the triangle
	inequality,
	\[
	\|(\mathbf v_h,q_h)\|_h
	\le
	\|(\mathbf u_h,p_h)\|_h
	+
	\delta_1\|(\mathbf v_{p,h},0)\|_h
	+
	\delta_2\|(0,r_h)\|_h.
	\]
	Lemma~4.8 gives
	\[
	\|(\mathbf v_{p,h},0)\|_h
	\le
	C
	\left(
	\mathcal P_h+\mathcal J_{1,h}+\mathcal J_{2,h}
	\right)^{1/2}
	\le
	C\|(\mathbf u_h,p_h)\|_h.
	\]
	Moreover, using Lemma~4.9,
	\[
	\|(0,r_h)\|_h^2
	=
	\|\alpha^{1/2}r_h\|_{0,\Omega_1\cup\Omega_2}^2
	+
	J_1(r_h,r_h)+J_2(r_h,r_h)
	\le
	C(\mathcal D_h+\mathcal G_h)
	\le
	C\|(\mathbf u_h,p_h)\|_h^2.
	\]
	Thus there exists \(C_T>0\), independent of \(h\) and of the cut position,
	such that
	\begin{equation}\label{eq:infsup-test-bound}
		\|(\mathbf v_h,q_h)\|_h
		\le
		C_T\|(\mathbf u_h,p_h)\|_h.
	\end{equation}
	
	Finally, by \eqref{eq:infsup-positive-test} and
	\eqref{eq:infsup-test-bound},
	\[
	\sup_{(\mathbf w_h,s_h)\in\mathbf V_h\times Q_h}
	\frac{
		\mathcal L_h(\mathbf u_h,p_h;\mathbf w_h,s_h)
	}{
		\|(\mathbf w_h,s_h)\|_h
	}
	\ge
	\frac{
		\mathcal L_h(\mathbf u_h,p_h;\mathbf v_h,q_h)
	}{
		\|(\mathbf v_h,q_h)\|_h
	}
	\ge
	\frac{C_I}{C_T}\|(\mathbf u_h,p_h)\|_h.
	\]
	Taking \(C_s=C_I/C_T\) proves the discrete inf-sup condition.
\end{proof}

\section{Error estimates}

\subsection{Interpolation estimate}
We use the extension operators and interpolants introduced in \eqref{early interpolation notation}. In particular, $\Idiv$ is the canonical $H(\operatorname{div})$ VEM interpolant and $\Pzero$ is the elementwise $L^2$ projection onto $\mathbb P_0$.

\begin{lem}[Pressure interpolation estimate]\label{lem:scalar-projection-J}
	Let \(p_I\) be defined by
	\[
	p_I|_K = P_K^0(E^i p_i)|_K,
	\qquad K\in\mathcal T_{h,i}^{+},\quad i=1,2,
	\]
	where \(P_K^0\) is the \(L^2(K)\)-projection onto \(\mathbb P_0(K)\).
	Assume that \(p_i\in H^1(\Omega_i)\), \(i=1,2\), and that the exact pressure
	satisfies
	\[
	[\![p]\!]=0 \qquad \text{on }\Gamma .
	\]
	Then
	\[
	J_1(p-p_I,p-p_I)+J_2(p-p_I,p-p_I)
	\le
	C h^2
	\|\alpha^{1/2}p\|_{1,\Omega_1\cup\Omega_2}^2 .
	\]
\end{lem}
\begin{proof}
	We first estimate the \(J_1\)-term. Let
	\(e\in\mathcal E_{\Gamma,i}\) be a transmission edge associated with the
	subdomain \(\Omega_i\), and let \(K^+\) and \(K^-\) be the two background
	elements sharing \(e\). Since the extended function \(E^ip_i\) is single-valued
	across \(e\), we have
	\[
	[\![E^ip_i]\!]_e=0 .
	\]
	Therefore,
	\[
	[\![p-p_I]\!]_e
	=
	[\![E^ip_i-p_I]\!]_e,
	\]
	and hence
	\[
	\|[\![p-p_I]\!]_e\|_{0,e}
	=
	\|[\![E^ip_i-p_I]\!]_e\|_{0,e}.
	\]
	By the triangle inequality,
	\[
	\|[\![E^ip_i-p_I]\!]_e\|_{0,e}^2
	\le
	C
	\sum_{K\in\omega_e}
	\|E^ip_i-P_K^0(E^ip_i)\|_{0,e}^2,
	\]
	where \(\omega_e=\{K^+,K^-\}\). Using the trace inequality on each
	background element \(K\in\omega_e\), we obtain
	\[
	\begin{aligned}
		h_e
		\|E^ip_i-P_K^0(E^ip_i)\|_{0,e}^2
		&\le
		C\left(
		\|E^ip_i-P_K^0(E^ip_i)\|_{0,K}^2
		+
		h_K^2 |E^ip_i|_{1,K}^2
		\right).
	\end{aligned}
	\]
	The approximation property of the \(L^2(K)\)-projection gives
	\[
	\|E^ip_i-P_K^0(E^ip_i)\|_{0,K}
	\le
	C h_K |E^ip_i|_{1,K}.
	\]
	Consequently,
	\[
	h_e
	\|[\![p-p_I]\!]_e\|_{0,e}^2
	\le
	C h^2
	\sum_{K\in\omega_e}
	|E^ip_i|_{1,K}^2 .
	\]
	Multiplying by \(\alpha_i\), summing over all
	\(e\in\mathcal E_{\Gamma,i}\), and using the finite overlap of the edge
	patches, we get
	\[
	\begin{aligned}
		J_1(p-p_I,p-p_I)
		&=
		\sum_{i=1}^2
		\sum_{e\in\mathcal E_{\Gamma,i}}
		h_e\alpha_i
		\|[\![p-p_I]\!]_e\|_{0,e}^2      \\
		&\le
		C h^2
		\sum_{i=1}^2
		\alpha_i
		\|E^ip_i\|_{1,\Omega_{h,i}}^2 .
	\end{aligned}
	\]
	By the extension estimate,
	\[
	\|E^ip_i\|_{1,\Omega_{h,i}}
	\le
	C\|p_i\|_{1,\Omega_i},
	\]
	and therefore
	\[
	J_1(p-p_I,p-p_I)
	\le
	C h^2
	\|\alpha^{1/2}p\|_{1,\Omega_1\cup\Omega_2}^2 .
	\]
	
	We next estimate the \(J_2\)-term. Since the exact pressure satisfies
	\[
	[\![p]\!]=0
	\qquad \text{on }\Gamma,
	\]
	we have, on each interface segment \(\Gamma_K\),
	\[
	[\![p-p_I]\!]
	=
	-[\![p_I-p]\!].
	\]
	Hence
	\[
	\|[\![p-p_I]\!]\|_{0,\Gamma_K}
	=
	\|[\![p_I-p]\!]\|_{0,\Gamma_K}.
	\]
	By the triangle inequality,
	\[
	\|[\![p_I-p]\!]\|_{0,\Gamma_K}^2
	\le
	C
	\sum_{i=1}^2
	\|p_{I,i}-p_i\|_{0,\Gamma_K}^2 .
	\]
	On \(K\in\mathcal T_h^\Gamma\), we have
	\[
	p_{I,i}=P_K^0(E^ip_i),
	\]
	and \(p_i\) on \(\Gamma_K\) is the trace of \(E^ip_i\). Thus the cut trace
	inequality gives
	\[
	\begin{aligned}
		h_K
		\|p_i-p_{I,i}\|_{0,\Gamma_K}^2
		&\le
		C\left(
		\|E^ip_i-P_K^0(E^ip_i)\|_{0,K}^2
		+
		h_K^2 |E^ip_i|_{1,K}^2
		\right)        \\
		&\le
		C h_K^2 |E^ip_i|_{1,K}^2 .
	\end{aligned}
	\]
	Therefore,
	\[
	\begin{aligned}
		J_2(p-p_I,p-p_I)
		&=
		\sum_{K\in\mathcal T_h^\Gamma}
		h_K\alpha_\Gamma
		\|[\![p-p_I]\!]\|_{0,\Gamma_K}^2          \\
		&\le
		C
		\sum_{K\in\mathcal T_h^\Gamma}
		\sum_{i=1}^2
		h_K\alpha_\Gamma
		\|p_i-p_{I,i}\|_{0,\Gamma_K}^2             \\
		&\le
		C h^2
		\sum_{K\in\mathcal T_h^\Gamma}
		\sum_{i=1}^2
		\alpha_\Gamma
		|E^ip_i|_{1,K}^2 .
	\end{aligned}
	\]
	Since
	\[
	\alpha_\Gamma
	=
	\frac{2\alpha_1\alpha_2}{\alpha_1+\alpha_2}
	\le
	2\alpha_i,
	\qquad i=1,2,
	\]
	we obtain
	\[
	J_2(p-p_I,p-p_I)
	\le
	C h^2
	\sum_{i=1}^2
	\alpha_i
	\|E^ip_i\|_{1,\Omega_{h,i}}^2 .
	\]
	Using again the extension estimate,
	\[
	J_2(p-p_I,p-p_I)
	\le
	C h^2
	\|\alpha^{1/2}p\|_{1,\Omega_1\cup\Omega_2}^2 .
	\]
	
	Combining the estimates for \(J_1\) and \(J_2\) proves the result.
\end{proof}

\begin{thm}\label{interpolation error estimation}
	Assume the standard local approximation estimates of the $H(\operatorname{div})$ VEM interpolant. For $\mathbf u\in H^2(\Omega_1\cup\Omega_2)^2$ with $\operatorname{div}\mathbf u\in H^1(\Omega_1\cup\Omega_2)$ and $p\in H^1(\Omega_1\cup\Omega_2)$ satisfying $\jp{p}=0$ on $\Gamma$, the interpolants defined in \eqref{early interpolation notation} satisfy
	\begin{equation}\label{interpolation estimate final}
		\Vert(\mathbf u-\mathbf u_I,p-p_I)\Vert_h
		\le Ch\left(
		\Vert\alpha^{-\frac12}\mathbf u\Vert_{2,\Omega_1\cup\Omega_2}
		+\Vert\alpha^{-\frac12}\operatorname{div}\mathbf u\Vert_{1,\Omega_1\cup\Omega_2}
		+\Vert\alpha^{\frac12}p\Vert_{1,\Omega_1\cup\Omega_2}
		\right).
	\end{equation}
\end{thm}

\begin{proof}
	Set $\bm\zeta=\mathbf u-\mathbf u_I$ and $\eta=p-p_I$. The projected $L^2$ term, the standard part of the VEM stabilization, and the physical divergence term are estimated by the standard VEM interpolation estimates and the commuting property:
	\begin{align*}
		&\sum_{i,K}\Vert\alpha_i^{-1/2}\QK\bm\zeta_i\Vert_{0,K_i}^2
		+\sum_{i,K}S_{\alpha,i}^K((I-\QK)\bm\zeta_i,(I-\QK)\bm\zeta_i)
		+\Vert\alpha^{-1/2}\operatorname{div}\bm\zeta\Vert_{0,\Omega_1\cup\Omega_2}^2 \\
		&\qquad\le Ch^2\left(
		\Vert\alpha^{-1/2}\mathbf u\Vert_{2,\Omega_1\cup\Omega_2}^2
		+\Vert\alpha^{-1/2}\operatorname{div}\mathbf u\Vert_{1,\Omega_1\cup\Omega_2}^2
		\right).
	\end{align*}
	On cut elements $S_{\alpha,i}^{K,+}=S_{\alpha,i}^{K}+G_{\operatorname{ker}}^K$ on the kernel part, and the additional contribution is exactly the first estimate in Lemma~\ref{lem interpolation new penalties}. The corrected flux penalty is controlled by the second estimate in the same lemma. The divergence ghost term satisfies
	\begin{equation*}
		\Gdiv(\bm\zeta,\bm\zeta)
		\le Ch^2\Vert\alpha^{-1/2}\operatorname{div}\mathbf u\Vert_{1,\Omega_1\cup\Omega_2}^2,
	\end{equation*}
	because the subdomain-wise extension is single-valued on transmission edges and the jump is the jump of the projection error of $\operatorname{div}\mathbf u$. Finally,
	\begin{equation*}
		\Vert\alpha^{1/2}\eta\Vert_{0,\Omega_1\cup\Omega_2}^2
		+J_1(\eta,\eta)+J_2(\eta,\eta)
		\le Ch^2\Vert\alpha^{1/2}p\Vert_{1,\Omega_1\cup\Omega_2}^2
	\end{equation*}
	by Lemma~\ref{lem:scalar-projection-J}. Combining these bounds gives \eqref{interpolation estimate final}.
\end{proof}

\subsection{Consistency and the a priori error estimate}

For later use define the regularity quantity
\begin{equation}\label{regularity N}
	\mathcal N(\mathbf u,p):=
	\Vert\alpha^{-1/2}\mathbf u\Vert_{2,\Omega_1\cup\Omega_2}
	+\Vert\alpha^{-1/2}\operatorname{div}\mathbf u\Vert_{1,\Omega_1\cup\Omega_2}
	+\Vert\alpha^{1/2}p\Vert_{1,\Omega_1\cup\Omega_2}.
\end{equation}

\begin{lem}[Consistency residual]\label{lem consistency residual}
	Let $(\mathbf u,p)$ be the exact solution and let $(\mathbf u_I,p_I)$ be defined by \eqref{early interpolation notation}. Then, for all $(\mathbf v_h,q_h)\in\mathbf V_h\times Q_h$,
	\begin{equation}\label{consistency residual estimate}
		\left|\mathcal L_h(\mathbf u_I,p_I;\mathbf v_h,q_h)
		-\sum_{i=1}^2\sum_{K\in\mathcal T_{h,i}^{+}}(f,q_{h,i})_{K_i}\right|
		\le Ch\,\mathcal N(\mathbf u,p)\Vert(\mathbf v_h,q_h)\Vert_h.
	\end{equation}
\end{lem}
\begin{proof}
	Set
	\[
	\mathcal N(\mathbf u,p)
	:=
	\|\alpha^{-1/2}\mathbf u\|_{2,\Omega_1\cup\Omega_2}
	+
	\|\alpha^{-1/2}\operatorname{div}\mathbf u\|_{1,\Omega_1\cup\Omega_2}
	+
	\|\alpha^{1/2}p\|_{1,\Omega_1\cup\Omega_2}.
	\]
	For \((\mathbf v_h,q_h)\in\mathbf V_h\times Q_h\), define
	\[
	\mathcal R_h(\mathbf v_h,q_h)
	:=
	\mathcal L_h(\mathbf u_I,p_I;\mathbf v_h,q_h)
	-
	\sum_{i=1}^2\sum_{K\in\mathcal T_{h,i}^{+}}(f,q_{h,i})_{K_i}.
	\]
	All sums over \(i,K\) below are understood as
	\[
	\sum_{i,K}
	:=
	\sum_{i=1}^2
	\sum_{K\in\mathcal T_{h,i}^{+}} .
	\]
	
	Using the continuous weak formulation, the definition of \(b_h\), and the
	interface conditions
	\[
	[\![p]\!]=0,\qquad [\![\mathbf u\cdot\mathbf n]\!]=0
	\quad\text{on }\Gamma,
	\]
	we obtain the following residual decomposition:
	\[
	\begin{aligned}
		\mathcal R_h(\mathbf v_h,q_h)
		&=
		\left[
		a_h(\mathbf u_I,\mathbf v_h)
		-
		\sum_{i,K}
		(\alpha_i^{-1}\mathbf u_i,\mathbf v_{h,i})_{K_i}
		\right]                                                \\
		&\quad
		+
		\sum_{i,K}
		(\operatorname{div}\mathbf v_{h,i},p_{I,i}-p_i)_{K_i}  \\
		&\quad
		-
		\sum_{K\in\mathcal T_h^\Gamma}
		([\![\mathcal M_\Gamma\mathbf v_h]\!],\{p_I-p\})_{\Gamma_K} \\
		&\quad
		+
		\mathcal C_\Gamma(p;\mathbf v_h)                       \\
		&\quad
		-
		\sum_{i,K}
		(\operatorname{div}(\mathbf u_{I,i}-\mathbf u_i),q_{h,i})_{K_i} \\
		&\quad
		+
		\sum_{K\in\mathcal T_h^\Gamma}
		([\![\mathcal M_\Gamma(\mathbf u_I-\mathbf u)]\!],\{q_h\})_{\Gamma_K} \\
		&\quad
		+
		\gamma_1J_1(p_I,q_h)
		+
		\gamma_2J_2(p_I,q_h),
	\end{aligned}
	\]
	where
	\[
	\mathcal C_\Gamma(p;\mathbf v_h)
	:=
	\sum_{K\in\mathcal T_h^\Gamma}
	\bigl(
	(I-\Pi_{\Gamma,K}^0)p,\,
	[\![\mathbf v_h\cdot\mathbf n]\!]
	-
	[\![\mathcal M_\Gamma\mathbf v_h]\!]
	\bigr)_{\Gamma_K}.
	\]
	Indeed, since
	\[
	[\![\mathcal M_\Gamma\mathbf v_h]\!]
	=
	\Pi_{\Gamma,K}^0[\![\mathbf v_h\cdot\mathbf n]\!],
	\]
	the difference
	\[
	[\![\mathbf v_h\cdot\mathbf n]\!]
	-
	[\![\mathcal M_\Gamma\mathbf v_h]\!]
	\]
	has zero average on \(\Gamma_K\). Hence the constant part
	\(\Pi_{\Gamma,K}^0p\) gives no contribution, and only
	\((I-\Pi_{\Gamma,K}^0)p\) remains in \(\mathcal C_\Gamma\).
	
	We now estimate the terms in this decomposition.
	
	First, by the commuting property of the \(H(\operatorname{div})\)-VEM
	interpolant,
	\[
	\operatorname{div}(\mathbf u_{I,i}-\mathbf u_i)
	=
	P_K^0(\operatorname{div}E^i_{\rm div}\mathbf u_i)
	-
	\operatorname{div}E^i_{\rm div}\mathbf u_i
	\qquad\text{on }K\in\mathcal T_{h,i}^{+}.
	\]
	Therefore, by the \(L^2\)-projection estimate,
	\[
	\left|
	\sum_{i,K}
	(\operatorname{div}(\mathbf u_{I,i}-\mathbf u_i),q_{h,i})_{K_i}
	\right|
	\le
	Ch
	\|\alpha^{-1/2}\operatorname{div}\mathbf u\|_{1,\Omega_1\cup\Omega_2}
	\|\alpha^{1/2}q_h\|_{0,\Omega_1\cup\Omega_2}.
	\]
	Moreover, using the consistency of \(\mathcal M_\Gamma\), Cauchy's
	inequality, the trace estimate for the pressure average, and Lemma~4.12,
	\[
	\left|
	\sum_{K\in\mathcal T_h^\Gamma}
	([\![\mathcal M_\Gamma(\mathbf u_I-\mathbf u)]\!],\{q_h\})_{\Gamma_K}
	\right|
	\le
	Ch
	\left(
	\|\alpha^{-1/2}\mathbf u\|_{2}
	+
	\|\alpha^{-1/2}\operatorname{div}\mathbf u\|_{1}
	\right)
	\|q_h\|_{Q,h}.
	\]
	Hence the two terms involving \(q_h\) are bounded by
	\[
	Ch\mathcal N(\mathbf u,p)\|(\mathbf v_h,q_h)\|_h.
	\]
	
	Next we estimate the pressure stabilization terms. Since \(E^ip_i\) is
	single-valued across the transmission edges associated with \(\Omega_i\),
	and since the exact pressure satisfies \([\![p]\!]=0\) on \(\Gamma\), we have
	\[
	J_1(p_I,q_h)=J_1(p_I-p,q_h),
	\qquad
	J_2(p_I,q_h)=J_2(p_I-p,q_h).
	\]
	Thus, by Cauchy's inequality and Lemma~5.1,
	\[
	|\gamma_1J_1(p_I,q_h)+\gamma_2J_2(p_I,q_h)|
	\le
	Ch\|\alpha^{1/2}p\|_{1,\Omega_1\cup\Omega_2}
	\|q_h\|_{Q,h}.
	\]
	
	We now estimate the two pressure projection terms in the first equation:
	\[
	\sum_{i,K}
	(\operatorname{div}\mathbf v_{h,i},p_{I,i}-p_i)_{K_i}
	-
	\sum_{K\in\mathcal T_h^\Gamma}
	([\![\mathcal M_\Gamma\mathbf v_h]\!],\{p_I-p\})_{\Gamma_K}.
	\]
	Since
	\[
	p_I|_K=P_K^0(E^ip_i)|_K
	\qquad\text{on }K\in\mathcal T_{h,i}^{+},
	\]
	the \(L^2\)-projection estimate gives
	\[
	\|p_I-p_i\|_{0,K_i}
	\le
	Ch_K|E^ip_i|_{1,K}.
	\]
	Together with the cut trace estimate for the interface term and the
	control of \([\![\mathcal M_\Gamma\mathbf v_h]\!]\) by the velocity norm,
	this yields
	\[
	\begin{aligned}
		&\left|
		\sum_{i,K}
		(\operatorname{div}\mathbf v_{h,i},p_{I,i}-p_i)_{K_i}
		\right|
		+
		\left|
		\sum_{K\in\mathcal T_h^\Gamma}
		([\![\mathcal M_\Gamma\mathbf v_h]\!],\{p_I-p\})_{\Gamma_K}
		\right|                                      \\
		&\qquad
		\le
		Ch\|\alpha^{1/2}p\|_{1,\Omega_1\cup\Omega_2}\|\mathbf v_h\|_h .
	\end{aligned}
	\]
	The remaining interface consistency term is exactly the quantity estimated
	in Lemma~4.11:
	\[
	|\mathcal C_\Gamma(p;\mathbf v_h)|
	\le
	Ch
	\|\alpha^{-1/2}\mathbf u\|_{2,\Omega_1\cup\Omega_2}
	\|\mathbf v_h\|_h .
	\]
	
	It remains to estimate the velocity consistency term
	\[
	a_h(\mathbf u_I,\mathbf v_h)
	-
	\sum_{i,K}
	(\alpha_i^{-1}\mathbf u_i,\mathbf v_{h,i})_{K_i}.
	\]
	We treat its components according to the definition of \(a_h\). The
	divergence ghost penalty satisfies
	\[
	|G_{\rm div}(\mathbf u_I,\mathbf v_h)|
	=
	|G_{\rm div}(\mathbf u_I-\mathbf u,\mathbf v_h)|
	\le
	Ch
	\|\alpha^{-1/2}\operatorname{div}\mathbf u\|_{1,\Omega_1\cup\Omega_2}
	\|\mathbf v_h\|_h,
	\]
	where we used the fact that \(\operatorname{div}\mathbf u\) is single-valued
	on the extended transmission edges after extension, together with the
	commuting property.
	
	For the corrected interface-flux penalty, we write
	\[
	P_{\mathcal F}(\mathbf u_I,\mathbf v_h)
	=
	P_{\mathcal F}(\mathbf u_I-\mathbf u,\mathbf v_h)
	+
	P_{\mathcal F}(\mathbf u,\mathbf v_h).
	\]
	The first term is bounded by Lemma~4.12. For the second one, the exact
	flux continuity gives
	\[
	[\![\mathcal M_\Gamma\mathbf u]\!]=0,
	\]
	and
	\[
	(I-\Pi_{\Gamma,K}^0)
	[\![\mathbf Q_K^1\mathbf u\cdot\mathbf n]\!]
	=
	(I-\Pi_{\Gamma,K}^0)
	[\![(\mathbf Q_K^1\mathbf u-\mathbf u)\cdot\mathbf n]\!].
	\]
	Thus the projection estimate for \(\mathbf Q_K^1\) and the trace inequality
	give
	\[
	|P_{\mathcal F}(\mathbf u_I,\mathbf v_h)|
	\le
	Ch
	\left(
	\|\alpha^{-1/2}\mathbf u\|_{2}
	+
	\|\alpha^{-1/2}\operatorname{div}\mathbf u\|_{1}
	\right)
	\|\mathbf v_h\|_h .
	\]
	
	It remains to consider the projected mass and the VEM stabilization parts.
	On each element \(K\), choose a polynomial
	\(\mathbf u_{\pi,i}\in[\mathbb P_1(K)]^2\) satisfying the usual
	approximation estimate. On non-cut elements, the standard VEM polynomial
	consistency and stability yield
	\[
	Ch\|\alpha^{-1/2}\mathbf u\|_{2,\Omega_1\cup\Omega_2}\|\mathbf v_h\|_h .
	\]
	On cut elements, the whole-cell orthogonality of \(\mathbf Q_K^1\) does not
	eliminate the virtual-kernel part on \(K_i\). The only nonstandard term is
	\[
	\sum_{K\in\mathcal T_h^\Gamma}\sum_{i=1}^2
	(\alpha_i^{-1}\mathbf u_{\pi,i},
	(I-\mathbf Q_K^1)\mathbf v_{h,i})_{K_i}.
	\]
	It is controlled by the cut-layer estimate and the enhanced kernel
	stabilization:
	\[
	\begin{aligned}
		&\left|
		\sum_{K\in\mathcal T_h^\Gamma}\sum_{i=1}^2
		(\alpha_i^{-1}\mathbf u_{\pi,i},
		(I-\mathbf Q_K^1)\mathbf v_{h,i})_{K_i}
		\right|                                      \\
		&\qquad
		\le
		\|\alpha^{-1/2}\mathbf u_{\pi}\|_{0,\Omega_h^\Gamma}
		\,
		\|\alpha^{-1/2}(I-\mathbf Q_K^1)\mathbf v_h\|_{0,\Omega_h^\Gamma}.
	\end{aligned}
	\]
	The strip estimate for the cut layer gives
	\[
	\|\alpha^{-1/2}\mathbf u_{\pi}\|_{0,\Omega_h^\Gamma}
	\le
	Ch^{1/2}
	\|\alpha^{-1/2}\mathbf u\|_{2,\Omega_1\cup\Omega_2},
	\]
	and the enhanced cut-element stabilization gives
	\[
	\|\alpha^{-1/2}(I-\mathbf Q_K^1)\mathbf v_h\|_{0,\Omega_h^\Gamma}
	\le
	Ch^{1/2}\|\mathbf v_h\|_h .
	\]
	Therefore the cut-cell residual is bounded by
	\[
	Ch\|\alpha^{-1/2}\mathbf u\|_{2,\Omega_1\cup\Omega_2}\|\mathbf v_h\|_h .
	\]
	Combining the non-cut and cut-element estimates gives
	\[
	\left|
	a_h(\mathbf u_I,\mathbf v_h)
	-
	\sum_{i,K}
	(\alpha_i^{-1}\mathbf u_i,\mathbf v_{h,i})_{K_i}
	\right|
	\le
	Ch
	\left(
	\|\alpha^{-1/2}\mathbf u\|_{2}
	+
	\|\alpha^{-1/2}\operatorname{div}\mathbf u\|_{1}
	\right)
	\|\mathbf v_h\|_h .
	\]
	
	Collecting all the estimates above, we obtain
	\[
	|\mathcal R_h(\mathbf v_h,q_h)|
	\le
	Ch\mathcal N(\mathbf u,p)
	\|(\mathbf v_h,q_h)\|_h .
	\]
	This proves the lemma.
\end{proof}

\begin{thm}[A priori error estimate]\label{thm final error}
	Assume the interface geometry in Assumption~\ref{ass unfitted geometry}. Let $(\mathbf u,p)$ be the exact solution of \eqref{weak formulation of mixed problem} with
	\[
	\mathbf u\in H^2(\Omega_1\cup\Omega_2)^2,
	\qquad
	\operatorname{div}\mathbf u\in H^1(\Omega_1\cup\Omega_2),
	\qquad
	p\in H^1(\Omega_1\cup\Omega_2),
	\]
	and let $(\mathbf u_h,p_h)$ be the solution of \eqref{discrete problem}. Then
	\begin{equation}\label{final error estimate}
		\Vert(\mathbf u-\mathbf u_h,p-p_h)\Vert_h
		\le Ch\left(
		\Vert\alpha^{-1/2}\mathbf u\Vert_{2,\Omega_1\cup\Omega_2}
		+\Vert\alpha^{-1/2}\operatorname{div}\mathbf u\Vert_{1,\Omega_1\cup\Omega_2}
		+\Vert\alpha^{1/2}p\Vert_{1,\Omega_1\cup\Omega_2}
		\right).
	\end{equation}
	The constant $C$ is independent of $h$ and of the position of the interface. In the present weighting it may depend on the contrast $\alpha_{\max}/\alpha_{\min}$.
\end{thm}
\begin{proof}
	By the triangle inequality,
	\[
	\Vert(\mathbf u-\mathbf u_h,p-p_h)\Vert_h
	\le \Vert(\mathbf u-\mathbf u_I,p-p_I)\Vert_h
	+\Vert(\mathbf u_I-\mathbf u_h,p_I-p_h)\Vert_h.
	\]
	The first term is bounded by Theorem~\ref{interpolation error estimation}. For the second term, the discrete inf-sup condition of Theorem~\ref{thm inf sup} gives
	\[
	\Vert(\mathbf u_I-\mathbf u_h,p_I-p_h)\Vert_h
	\le C\sup_{(\mathbf v_h,q_h)\ne0}
	\frac{\mathcal L_h(\mathbf u_I-\mathbf u_h,p_I-p_h;\mathbf v_h,q_h)}{\Vert(\mathbf v_h,q_h)\Vert_h}.
	\]
	Using the discrete problem \eqref{discrete problem}, the numerator is exactly the consistency residual in Lemma~\ref{lem consistency residual}. Hence it is bounded by $Ch\mathcal N(\mathbf u,p)$. This proves \eqref{final error estimate}.
\end{proof}

\section{Numerical experiments}\label{sec numerical}

In this section we report numerical experiments for the proposed method. The computations are performed on unfitted polygonal meshes. Unless otherwise stated, we use
\[
\gamma_1=1,\qquad \gamma_2=1,\qquad \gamma_{\mathcal F}=10,
\qquad \eta_{\rm ker}=1,
\qquad \eta_{\rm div}=0.1 .
\]
For circular interfaces the intersection points and the interface quadrature are computed on the circular arc. The reported rates are computed from two consecutive meshes.

For nonhomogeneous boundary data $p=g_D$ on $\partial\Omega$, the right-hand side of the first mixed equation is modified by adding the boundary contribution $\langle g_D,\mathbf v_h\cdot\mathbf n\rangle_{\partial\Omega}$. This is the only modification required for the examples with nonzero boundary data.

We use the following relative errors. The relative mesh-dependent error is evaluated with the true exact solution,
\begin{equation}\label{numerical mesh error definition}
	E_h^{\rm rel}
	=
	\frac{\| (\mathbf u-\mathbf u_h,p-p_h)\|_{h,{\rm ex}}}
	{\| (\mathbf u,p)\|_{h,{\rm ex}}}.
\end{equation}
Here the projected, divergence, pressure, interface-flux, and ghost-penalty components are evaluated with the same weights as in the discrete norm. The VEM stabilization part is evaluated from the scaled degrees of freedom of the exact velocity, equivalently from the canonical VEM interpolant of the exact velocity. We also report the relative projected velocity error
\begin{equation}\label{numerical projected velocity error}
	E_{u,Q}^{\rm rel}
	=
	\frac{
		\left(\sum_{i=1}^2\sum_{K\in\mathcal T_{h,i}^+}
		\|\mathbf u_i-\QK\mathbf u_{h,i}\|_{0,K_i}^2\right)^{1/2}}
	{\|\mathbf u\|_{0,\Omega_1\cup\Omega_2}},
\end{equation}
The pressure contribution is already included in $E_h^{\rm rel}$ through the pressure part of the mesh-dependent norm and the pressure stabilization terms. Therefore, in the tables below we report $E_h^{\rm rel}$ as the main mixed error and $E_{u,Q}^{\rm rel}$ as an additional measure of the accuracy of the computed flux.

\subsection{Example 1: straight interface}

Let
\[
\phi(x,y)=x-0.3,
\qquad
\Omega_1=\{x>0.3\},\qquad \Omega_2=\{x\le 0.3\}.
\]
Set
\[
a=x-0.3,
\qquad
w=a^2(1-y^2).
\]
The exact pressure and velocity are
\[
p(x,y)=
\begin{cases}
	w/\alpha_1, & \phi(x,y)>0,\\[0.3em]
	w/\alpha_2, & \phi(x,y)\le 0,
\end{cases}
\qquad
\mathbf u=\nabla w
=
\begin{pmatrix}
	2(x-0.3)(1-y^2)\\
	-2y(x-0.3)^2
\end{pmatrix}.
\]
Thus
\[
f=-\operatorname{div}\mathbf u
=2\bigl((x-0.3)^2+y^2-1\bigr),
\qquad
g_D=p|_{\partial\Omega}.
\]
The convergence results are shown in Table~\ref{tab example1}.

\begin{table}[H]
	\centering
	\caption{Example 1: straight interface with nonhomogeneous boundary data.}
	\label{tab example1}
	\small
	\setlength{\tabcolsep}{6pt}
	\renewcommand{\arraystretch}{1.18}
	\begin{tabular}{c c c c c c c}
		\hline
		$(\alpha_1,\alpha_2)$ & $N_{\rm el}$ & $h$ & $E_h^{\rm rel}$ & rate & $E_{u,Q}^{\rm rel}$ & rate \\
		\hline
		\multirow{4}{*}{$(1,10^5)$}
		& 256   & 0.182 & $8.064\mathrm e{-2}$ & --    & $6.055\mathrm e{-3}$ & -- \\
		& 1024  & 0.091 & $3.902\mathrm e{-2}$ & 1.046 & $1.471\mathrm e{-3}$ & 2.038 \\
		& 4096  & 0.045 & $1.913\mathrm e{-2}$ & 1.022 & $3.606\mathrm e{-4}$ & 2.017 \\
		& 16384 & 0.023 & $9.509\mathrm e{-3}$ & 1.019 & $8.994\mathrm e{-5}$ & 2.025 \\
		\hline
		\multirow{4}{*}{$(10^5,1)$}
		& 256   & 0.182 & $7.721\mathrm e{-2}$ & --    & $5.876\mathrm e{-3}$ & -- \\
		& 1024  & 0.091 & $3.811\mathrm e{-2}$ & 1.017 & $1.423\mathrm e{-3}$ & 2.043 \\
		& 4096  & 0.045 & $1.896\mathrm e{-2}$ & 1.001 & $3.486\mathrm e{-4}$ & 2.018 \\
		& 16384 & 0.023 & $9.466\mathrm e{-3}$ & 1.013 & $8.707\mathrm e{-5}$ & 2.022 \\
		\hline
	\end{tabular}
\end{table}

\subsection{Example 2: circular interface}

Let
\[
\phi(x,y)=x^2+y^2-r_0^2,
\qquad r_0=0.3,
\qquad r=(x^2+y^2)^{1/2}.
\]
We take
\[
\mathbf u=\nabla(r^5)=
\begin{pmatrix}5xr^3\\ 5yr^3\end{pmatrix},
\]
and
\[
p(x,y)=
\begin{cases}
	r^5/\alpha_2, & \phi(x,y)\le 0,\\[0.4em]
	r^5/\alpha_1+(\alpha_2^{-1}-\alpha_1^{-1})r_0^5, & \phi(x,y)>0.
\end{cases}
\]
Then $\jp{p}=0$ and $\jp{\mathbf u\cdot\mathbf n}=0$ on the circle $r=r_0$. Moreover
\[
f=-\operatorname{div}\mathbf u=-25r^3,
\qquad
g_D=p|_{\partial\Omega}.
\]
Table~\ref{tab example2} gives the convergence results.

\begin{table}[H]
	\centering
	\caption{Example 2: circular interface with nonhomogeneous boundary data.}
	\label{tab example2}
	\small
	\setlength{\tabcolsep}{6pt}
	\renewcommand{\arraystretch}{1.18}
	\begin{tabular}{c c c c c c c}
		\hline
		$(\alpha_1,\alpha_2)$ & $N_{\rm el}$ & $h$ & $E_h^{\rm rel}$ & rate & $E_{u,Q}^{\rm rel}$ & rate \\
		\hline
		\multirow{4}{*}{$(1,10^5)$}
		& 256   & 0.182 & $1.011\mathrm e{-1}$ & --    & $1.421\mathrm e{-2}$ & -- \\
		& 1024  & 0.091 & $5.021\mathrm e{-2}$ & 1.008 & $3.415\mathrm e{-3}$ & 2.054 \\
		& 4096  & 0.045 & $2.508\mathrm e{-2}$ & 0.996 & $8.463\mathrm e{-4}$ & 2.001 \\
		& 16384 & 0.023 & $1.255\mathrm e{-2}$ & 1.010 & $2.114\mathrm e{-4}$ & 2.022 \\
		\hline
		\multirow{4}{*}{$(10^5,1)$}
		& 256   & 0.182 & $5.347\mathrm e{-2}$ & --    & $1.660\mathrm e{-2}$ & -- \\
		& 1024  & 0.091 & $2.966\mathrm e{-2}$ & 0.849 & $4.085\mathrm e{-3}$ & 2.020 \\
		& 4096  & 0.045 & $1.435\mathrm e{-2}$ & 1.041 & $1.018\mathrm e{-3}$ & 1.993 \\
		& 16384 & 0.023 & $6.864\mathrm e{-3}$ & 1.076 & $2.543\mathrm e{-4}$ & 2.023 \\
		\hline
	\end{tabular}
\end{table}

For the pair $(\alpha_1,\alpha_2)=(10^5,1)$, the mesh-dependent error shows a short pre-asymptotic behavior on the coarsest refinement, while the last two rates are close to first order. The projected velocity error is close to second order throughout the refinements.

\subsection{Example 3: position of the interface}

This example tests the sensitivity with respect to the position of the interface relative to the background mesh. We consider a family of straight interfaces
\[
\Gamma_\xi:\quad x+\frac{\pi}{6}+\xi=0,
\qquad
\xi=10^{-1},10^{-2},\ldots,10^{-7}.
\]
Let
\[
\phi_\xi(x,y)=x+\frac{\pi}{6}+\xi,
\qquad
w_\xi=\phi_\xi(x,y)(x^2-1)(y^2-1).
\]
The exact solution is defined by
\[
p=
\begin{cases}
	w_\xi/\alpha_1, & \phi_\xi>0,\\[0.3em]
	w_\xi/\alpha_2, & \phi_\xi\le0,
\end{cases}
\qquad
\mathbf u=\nabla w_\xi,
\qquad
f=-\Delta w_\xi.
\]
Since $w_\xi=0$ on $\Gamma_\xi$ and on $\partial\Omega$, this example has $\jp{p}=0$, $\jp{\mathbf u\cdot\mathbf n}=0$, and homogeneous boundary data. The explicit velocity is
\[
\mathbf u=
\begin{pmatrix}
	(y^2-1)\left[(x^2-1)+2x(x+\frac{\pi}{6}+\xi)\right]\\[0.2em]
	2y(x+\frac{\pi}{6}+\xi)(x^2-1)
\end{pmatrix}.
\]
Let $N_{\rm cut}=\#\mathcal T_h^\Gamma$ and $\rho_{\rm min}=\min_{K\in\mathcal T_h^\Gamma,\ i=1,2}|K_i|/|K|$. We use $(\alpha_1,\alpha_2)=(10^5,1)$ on a fixed mesh with $N_{\rm el}=4096$ and $h=0.045$. Table~\ref{tab example3} shows that the errors remain essentially unchanged even when the smallest cut ratio becomes very small. This type of moving-interface test is commonly used to examine cut-position robustness of unfitted methods.

\begin{table}[H]
	\centering
	\caption{Example 3: sensitivity with respect to the interface position on a fixed mesh with $N_{\rm el}=4096$ and $h=0.045$.}
	\label{tab example3}
	\small
	\setlength{\tabcolsep}{6pt}
	\renewcommand{\arraystretch}{1.18}
	\begin{tabular}{c c c c c c}
		\hline
		$(\alpha_1,\alpha_2)$ & $\xi$ & $N_{\rm cut}$ & $\rho_{\rm min}$ & $E_h^{\rm rel}$ & $E_{u,Q}^{\rm rel}$ \\
		\hline
		\multirow{7}{*}{$(10^5,1)$}
		& $10^{-1}$ & 88 & $2.725\mathrm e{-4}$ & $2.060\mathrm e{-2}$ & $5.502\mathrm e{-4}$ \\
		& $10^{-2}$ & 78 & $6.005\mathrm e{-6}$ & $2.029\mathrm e{-2}$ & $5.901\mathrm e{-4}$ \\
		& $10^{-3}$ & 71 & $3.517\mathrm e{-5}$ & $2.015\mathrm e{-2}$ & $6.260\mathrm e{-4}$ \\
		& $10^{-4}$ & 69 & $7.859\mathrm e{-5}$ & $2.014\mathrm e{-2}$ & $6.402\mathrm e{-4}$ \\
		& $10^{-5}$ & 69 & $1.125\mathrm e{-5}$ & $2.014\mathrm e{-2}$ & $6.412\mathrm e{-4}$ \\
		& $10^{-6}$ & 69 & $7.854\mathrm e{-6}$ & $2.015\mathrm e{-2}$ & $6.413\mathrm e{-4}$ \\
		& $10^{-7}$ & 69 & $7.548\mathrm e{-6}$ & $2.015\mathrm e{-2}$ & $6.413\mathrm e{-4}$ \\
		\hline
	\end{tabular}
\end{table}

\subsection{Example 4: heart-shaped interface}

Finally we test a more complicated interface. Let
\[
\phi(x,y)=(3x^2+3y^2-x)^2-x^2-y^2+0.02.
\]
Set
\[
B(x,y)=(1-x^2)(1-y^2),
\qquad
w(x,y)=\phi(x,y)B(x,y).
\]
We define
\[
p=
\begin{cases}
	w/\alpha_1, & \phi>0,\\[0.3em]
	w/\alpha_2, & \phi\le0,
\end{cases}
\qquad
\mathbf u=\nabla w,
\qquad
f=-\Delta w.
\]
Since $B=0$ on $\partial\Omega$ and $w=0$ on $\Gamma$, the boundary condition is homogeneous and the interface conditions are satisfied. The convergence results are reported in Table~\ref{tab example4}.

\begin{table}[H]
	\centering
	\caption{Example 4: heart-shaped interface.}
	\label{tab example4}
	\small
	\setlength{\tabcolsep}{6pt}
	\renewcommand{\arraystretch}{1.18}
	\begin{tabular}{c c c c c c c}
		\hline
		$(\alpha_1,\alpha_2)$ & $N_{\rm el}$ & $h$ & $E_h^{\rm rel}$ & rate & $E_{u,Q}^{\rm rel}$ & rate \\
		\hline
		\multirow{4}{*}{$(1,10^5)$}
		& 256   & 0.182 & $1.951\mathrm e{-1}$ & --    & $6.719\mathrm e{-2}$ & -- \\
		& 1024  & 0.091 & $9.834\mathrm e{-2}$ & 0.987 & $1.687\mathrm e{-2}$ & 1.991 \\
		& 4096  & 0.045 & $4.899\mathrm e{-2}$ & 0.999 & $4.201\mathrm e{-3}$ & 1.994 \\
		& 16384 & 0.023 & $2.451\mathrm e{-2}$ & 1.010 & $1.052\mathrm e{-3}$ & 2.018 \\
		\hline
		\multirow{4}{*}{$(10^5,1)$}
		& 256   & 0.182 & $2.742\mathrm e{-1}$ & --    & $7.294\mathrm e{-2}$ & -- \\
		& 1024  & 0.091 & $1.247\mathrm e{-1}$ & 1.135 & $1.820\mathrm e{-2}$ & 2.000 \\
		& 4096  & 0.045 & $5.969\mathrm e{-2}$ & 1.063 & $4.541\mathrm e{-3}$ & 2.003 \\
		& 16384 & 0.023 & $2.816\mathrm e{-2}$ & 1.076 & $1.128\mathrm e{-3}$ & 2.011 \\
		\hline
	\end{tabular}
\end{table}

\subsection{Additional test: influence of the coefficient contrast}

Although the constants in the analysis may depend on the coefficient contrast, we also test additional coefficient pairs to examine the practical behavior of the method. The extreme pairs \((1,10^5)\) and \((10^5,1)\) have already been reported in Table~\ref{tab example2}. Table~\ref{tab contrast circle} reports
the results for the additional pairs \((1,10)\), \((10,1)\), \((1,10^3)\), and \((10^3,1)\) for the same circular interface example.

\begin{table}[H]
	\centering
	\caption{Example 2 with additional coefficient contrasts.}
	\label{tab contrast circle}
	\small
	\setlength{\tabcolsep}{5.5pt}
	\renewcommand{\arraystretch}{1.18}
	\begin{tabular}{c c c c c c c}
		\hline
		$(\alpha_1,\alpha_2)$ & $N_{\rm el}$ & $h$ 
		& $E_h^{\rm rel}$ & rate 
		& $E_{u,Q}^{\rm rel}$ & rate \\
		\hline
		\multirow{4}{*}{$(1,10)$}
		& 256   & 0.182 & $1.011\mathrm e{-1}$ & --    & $1.375\mathrm e{-2}$ & -- \\
		& 1024  & 0.091 & $5.022\mathrm e{-2}$ & 1.008 & $3.361\mathrm e{-3}$ & 2.029 \\
		& 4096  & 0.045 & $2.508\mathrm e{-2}$ & 0.996 & $8.346\mathrm e{-4}$ & 1.998 \\
		& 16384 & 0.023 & $1.255\mathrm e{-2}$ & 1.010 & $2.091\mathrm e{-4}$ & 2.018 \\
		\hline
		\multirow{4}{*}{$(10,1)$}
		& 256   & 0.182 & $1.012\mathrm e{-1}$ & --    & $1.649\mathrm e{-2}$ & -- \\
		& 1024  & 0.091 & $5.031\mathrm e{-2}$ & 1.007 & $4.076\mathrm e{-3}$ & 2.014 \\
		& 4096  & 0.045 & $2.513\mathrm e{-2}$ & 0.996 & $1.014\mathrm e{-3}$ & 1.996 \\
		& 16384 & 0.023 & $1.257\mathrm e{-2}$ & 1.010 & $2.536\mathrm e{-4}$ & 2.021 \\
		\hline
		\multirow{4}{*}{$(1,10^3)$}
		& 256   & 0.182 & $1.011\mathrm e{-1}$ & --    & $1.436\mathrm e{-2}$ & -- \\
		& 1024  & 0.091 & $5.021\mathrm e{-2}$ & 1.008 & $3.428\mathrm e{-3}$ & 2.064 \\
		& 4096  & 0.045 & $2.508\mathrm e{-2}$ & 0.996 & $8.453\mathrm e{-4}$ & 2.008 \\
		& 16384 & 0.023 & $1.255\mathrm e{-2}$ & 1.010 & $2.111\mathrm e{-4}$ & 2.023 \\
		\hline
		\multirow{4}{*}{$(10^3,1)$}
		& 256   & 0.182 & $1.149\mathrm e{-1}$ & --    & $1.656\mathrm e{-2}$ & -- \\
		& 1024  & 0.091 & $5.890\mathrm e{-2}$ & 0.963 & $4.089\mathrm e{-3}$ & 2.015 \\
		& 4096  & 0.045 & $2.914\mathrm e{-2}$ & 1.009 & $1.016\mathrm e{-3}$ & 1.997 \\
		& 16384 & 0.023 & $1.439\mathrm e{-2}$ & 1.029 & $2.541\mathrm e{-4}$ & 2.021 \\
		\hline
	\end{tabular}
\end{table}
Together with Table~\ref{tab example2}, these results show the expected convergence behavior for all tested coefficient pairs. This experiment is intended to illustrate the practical behavior of the method for several contrasts; the theoretical constants are not claimed to be independent of the coefficient contrast.

\section{Concluding remarks}

We have proposed a lowest-order \(H(\operatorname{div})\)-conforming mixed extended virtual element method for elliptic interface problems on unfitted
polygonal meshes. The method uses a whole-cell polynomial projection, an interface normal-flux average, a corrected interface-flux penalty, and a local divergence ghost stabilization. These ingredients allow one to keep the method computable on polygonal meshes and to control the virtual-kernel consistency terms on cut elements.

We proved a discrete inf-sup condition and an optimal first-order error estimate in a mesh-dependent norm. The constants are independent of the mesh size and of the position of the interface relative to the background mesh, although they may depend on the coefficient contrast. The numerical results confirm the predicted convergence rates and show stable behavior when the interface cuts the mesh in unfavorable positions.

\end{document}